\documentclass[10pt]{amsart}

\usepackage{tikz}
\usepackage{graphicx}
\usepackage{tikz-cd}
\usepackage[shortlabels]{enumitem}

\usetikzlibrary{decorations.pathmorphing}

\tikzset{snake it/.style={decorate, decoration=snake}}

\newtheorem{theorem}{Theorem}[section]
\newtheorem{lemma}[theorem]{Lemma}
\newtheorem{proposition}[theorem]{Proposition}

\theoremstyle{definition}
\newtheorem{definition}[theorem]{Definition}

\newtheorem{example}[theorem]{Example}

\theoremstyle{remark}
\newtheorem{remark}[theorem]{Remark}

\numberwithin{equation}{section}

\newcommand{\abs}[1]{\lvert#1\rvert}
\newcommand{\norm}[1]{\lvert\lvert#1\rvert\rvert}

\begin{document}

\title{Index bounded relative symplectic cohomology}

\author{Yuhan Sun}

\date{\today}

\address{Hill Center, Department of Mathematics, Rutgers University, 110 Frelinghuysen Rd, Piscataway NJ 08854.}

\email{sun.yuhan@rutgers.edu}

\begin{abstract}
We study the relative symplectic cohomology with the help of an index bounded contact form. For a Liouville domain with an index bounded boundary, we construct a spectral sequence which starts from its classical symplectic cohomology and converges to the relative symplectic cohomology of it inside a Calabi-Yau manifold.
\end{abstract}

\maketitle


\tableofcontents

\section{Introduction}

Given a closed symplectic manifold $(M, \omega)$ and a compact subset $D$ of $M$, the relative symplectic cohomology $SH_{M}(D)$ is a Floer-theoretic invariant, which captures both dynamical and topological information of the pair $(M, D)$. Its construction by Varolgunes \cite{Var} is based on the Hamiltonian Floer theory of $M$ with more algebraic ingredients. Roughly speaking, one considers a family of increasing Hamiltonian functions that go to zero on $D$ while going to positive infinity on $M-D$. Then $SH_{M}(D)$ is defined as the homology of a completed telescope of Floer complexes, given by this family of Hamiltonian functions.

The idea of using Hamiltonian functions ``localized at a subset'' may date back to Cieliebak-Floer-Hofer \cite{CFH} and Viterbo \cite{Vit1}. Recently there are several new versions of Hamiltonian Floer theories related to this idea and aimed at various local-to-global problems. Besides \cite{Var}, let us mention an incomplete list: Groman \cite{Gro}, McLean \cite{M2020} and Venkatesh \cite{Ven}. A priori, the definitions in \cite{Gro, M2020, Var, Ven} are different, depending on whether $M$ is open or closed, taking completion with the action filtration or the Novikov filtration, and the orders of taking different limits. It would be interesting to compare these theories to attack particular problems. But in this article we mainly focus on the version of $SH_{M}(D)$ by Varolgunes.

Along its definition in \cite{Var}, several good properties of this invariant $SH_{M}(D)$ have been established, including the Hamiltonian isotopy invariance, a K\"unneth formula, a displaceability criterion and the Mayer-Vietoris property. These properties indicate that this invariant would play an important role in symplectic topology and mirror symmetry. One motivation of it comes from mirror symmetry suggested by Seidel \cite{Se2} and the family Floer program. On the other hand, some symplectic topological applications have already appeared in \cite{TVar, DGPZ}. Also see its relation with quantum cohomology by Borman-Sheridan-Varolgunes \cite{BSV}.

Here we study a computational side of this invariant, and provide some applications to symplectic topology. The main goal is to construct a filtration on the underlying complex of this cohomology and look at the induced spectral sequence. Now we set up notations and state our results. Let $(M, \omega)$ be a closed symplectic manifold. We say $M$ is \textit{symplectic Calabi-Yau} if $c_{1}(TM)=0$. A \textit{convex domain} $(D, \theta)$ in $M$ is a compact codimension zero symplectic submanifold of $M$, with a boundary $\partial D$ and a one-form $\theta$ locally defined near $\partial D$, such that the restriction of $\theta$ to $\partial D$ is a contact form and the local Liouville vector field points outward. A \textit{Liouville domain} $(D, \theta)$ in $M$ is a convex domain in $M$ such that the one-form $\theta$ is defined on all of $D$ and the restriction of $\omega$ on $D$ is $d\theta$. We will focus on a special family of convex domains, where the restriction of $\theta$ to $\partial D$ are \textit{index bounded} contact forms, see Definition \ref{index}.

For a Liouville domain in $M$ we will equip it with an auxiliary form $\tilde{\omega}$ which represents a class in $H^{2}(M, D; \mathbb{R})$, see Lemma \ref{aux}. If $[\tilde{\omega}]$ has integral values on $H_{2}(M, D; \mathbb{Z})$, then we say it is integral. This auxiliary form $\tilde{\omega}$ will be used to characterize how far a Floer solution travels outside $D$.

Let $\Lambda_{0}$ be the Novikov ring and $\Lambda_{E}$ be the truncated Novikov ring, see Section \ref{sec:back} for the notations. Our main result is the following.

\begin{theorem}\label{main}
Let $(M, \omega)$ be a closed symplectic Calabi-Yau manifold and $D$ be a Liouville domain in $M$ with an index bounded boundary. Suppose that $[\tilde{\omega}]$ is integral. Given any positive number $E$, there is a truncated invariant $SH_{M}(D; \Lambda_{E})$ such that
\begin{enumerate}
\item There is a spectral sequence that starts from the classical symplectic cohomology $SH(D; \Lambda_{E})$ with coefficient $\Lambda_{E}$, and converges to $SH_{M}(D; \Lambda_{E})$.
\item If the class $[\tilde{\omega}]$ vanishes on $H_{2}(M,D)$, then the above spectral sequence degenerates at the first page, which shows that $SH_{M}(D; \Lambda_{E})\cong SH(D; \Lambda_{E})$.
\item For an increasing sequence $E_{1} <E_{2}< \cdots$ that goes to positive infinity, the inverse limit of the truncated invariant recovers the relative symplectic cohomology
$$
(\varprojlim_{E_{i}}SH^{k}_{M}(D; \Lambda_{E_{i}}))\otimes_{\Lambda_{0}}\Lambda\cong SH^{k}_{M}(D)\otimes_{\Lambda_{0}}\Lambda.
$$
\end{enumerate}
\end{theorem}

\begin{remark}
In the definition of an index bounded contact form, we assume that it is non-degenerate. We will give a perturbative method in Section \ref{sec:examples} that works for Morse-Bott non-degenerate contact forms.
\end{remark}

The proof of the above theorem draws much inspiration from the work \cite{M2020} of McLean and its usage of the index bounded condition. Now we sketch the proof. The Hamiltonian functions we are using to compute $SH_{M}(D)$ are approximately zero on $D$ and positive infinity outside $D$, as a direct limit. The non-constant periodic orbits of our interest lie around the boundary of $D$, see Figure \ref{Ham}. We call the integral of the Hamiltonian function over a periodic orbit the \textit{level of the orbit}. Then the levels of these orbits can go to either zero or positive infinity. We first define a valuation on the free module generated by these orbits. However, due to the completion procedure, there will be elements with a negative infinite valuation which comes from limits of orbits going to infinite high level. Then we use the index bounded condition to ignore these high limits of orbits. As a consequence, the original underlying complex of this relative invariant is quasi-isomorphic to a new complex without the high limits. And the same valuation on the new complex gives an exhaustive filtration, which will induce a convergent spectral sequence.

An important class of examples fitting into the above theorem comes from simply-connected Lagrangian submanifolds in Calabi-Yau manifolds. Let $L$ be a simply-connected Lagrangian submanifold in a Calabi-Yau manifold $M$. Take $D$ as a Weinstein neighborhood of $L$, which is isomorphic to a disk bundle $D_{r}T^{*}L$ of the cotangent bundle of $L$, with respect to some Riemannian metric $g$ on $L$. There is a correspondence between the geodesics of $g$ and the Reeb orbits on the co-sphere bundle of $T^{*}L$. Hence the index bounded condition for the contact form on the co-sphere bundle will be satisfied if the metric $g$ satisfies some relations between the length of closed geodesics and their Morse indices. For many simply-connected manifolds, the existence of such a nice Riemannian metric is known. (In Section \ref{sec:examples} we show it is true when $g$ has a positive Ricci curvature.) Then we obtain a spectral sequence starting from $SH(D; \Lambda_{E})$ and converging to $SH_{M}(D; \Lambda_{E})$. Note that $SH_{M}(D; \Lambda)\otimes_{\Lambda_{0}}\Lambda$ detects the displaceability of $D$ inside $M$ (Theorem \ref{t: stable}), and it does not depend on $r$ in the index bounded case (Proposition 1.13 in \cite{TVar}). Hence we can let $r\rightarrow 0$ to detect the displaceability of $L$ itself inside $M$. On the other hand, usual invariant to detect the displaceability of $L$, the self-Lagrangian Floer cohomology $HF(L)$ may not be defined due to possible holomorphic disks on $L$ with Maslov index zero. Moreover, by using the Mayer-Vietoris property, we can also study the complement of Lagrangian submanifolds. We present one sample application of Theorem \ref{main}.

\begin{proposition}
Let $(M, \omega)$ be a symplectic Calabi-Yau manifold with dimension greater than four and $\omega$ represents an integral class in $H^{2}(M)$. For a simply-connected Lagrangian $S$ in $M$ and a Weinstein neighborhood $U$ of $S$, we have that $M-U$ is not stably-displaceable in $M$.
\end{proposition}
\begin{proof}
(See a more detailed proof in Proposition \ref{comp}) By using our spectral sequence we can show that $SH^{2n}_{M}(U)\otimes_{\Lambda_{0}}\Lambda= 0$, where $2n$ is the dimension of $M$. Hence the result follows from the stable displacement criterion (Theorem \ref{t: stable}) and the Mayer-Vietoris property of the relative symplectic cohomology.
\end{proof}

\begin{remark}
This proposition can be regarded as an analogue of a result of Ishikawa: if $U$ is a round ball in a Calabi-Yau manifold $M$ then $M-U$ is not stably-displaceable, see Theorem 1.1 \cite{I} and Corollary 1.15 \cite{TVar}. Ishikawa's proof uses computations of spectral invariants of certain distance functions, which shows that $M-U$ is always a super-heavy set.
\end{remark}

The outline of this article is as follows. In Section \ref{sec:back} we review backgrounds about Hamiltonian Floer theories. In Section \ref{sec:fil} and \ref{sec:ss} we construct the filtration and show its properties to prove Theorem \ref{main}. In Section \ref{sec:examples} we discuss some extensions of the theorems as well as applications. 

\begin{remark}
We focus on the case that $M$ is Calabi-Yau and $D$ is index bounded. More local-to-global results of $SH_{M}(D)$ in other interesting cases can be found in the work by Borman-Sheridan-Varolgunes \cite{BSV} and by Groman-Varolgunes \cite{GroV}.
\end{remark}

\subsection*{Acknowledgments}
The author acknowledges Mark McLean for his generous guidance on this project. The author also acknowledges Umut Varolgunes for helpful discussions. 

\section{Backgrounds}\label{sec:back}
Now we review the construction of symplectic cohomology theories. First we specify the ring and field that will be used. The Novikov ring $\Lambda_{0}$ and its field $\Lambda$ of fractions are defined by
$$
\Lambda_{0}=\lbrace \sum_{i=0}^{\infty}a_{i}T^{\lambda_{i}}\mid a_{i}\in \mathbb{C}, \lambda_{i}\in\mathbb{R}_{\geq 0}, \lambda_{i}<\lambda_{i+1}, \lim_{i\rightarrow \infty}\lambda_{i}=+\infty \rbrace
$$
and
$$
\Lambda=\lbrace \sum_{i=0}^{\infty}a_{i}T^{\lambda_{i}}\mid a_{i}\in \mathbb{C}, \lambda_{i}\in\mathbb{R}, \lambda_{i}<\lambda_{i+1}, \lim_{i\rightarrow \infty}\lambda_{i}=+\infty \rbrace
$$
where $T$ is a formal variable. The maximal ideal of $\Lambda_{0}$ is defined by
$$
\Lambda_{+}=\lbrace \sum_{i=0}^{\infty}a_{i}T^{\lambda_{i}}\mid a_{i}\in \mathbb{C}, \lambda_{i}\in\mathbb{R}_{>0}, \lambda_{i}<\lambda_{i+1}, \lim_{i\rightarrow \infty}\lambda_{i}=+\infty \rbrace.
$$
There is a valuation $v: \Lambda \rightarrow \mathbb{R}\cup \lbrace +\infty \rbrace$ by setting
$$
v(\sum_{i=0}^{\infty}a_{i}T^{\lambda_{i}}):= \min_{i} \lbrace \lambda_{i}\mid a_{i}\neq 0\rbrace \quad \text{and} \quad v(0):=+\infty
$$
which makes $\Lambda_{0}$ to be a complete valuation ring. When we say the completion of a $\Lambda_{0}$-module we mean the completion with respect to this valuation. We write
$$
\Lambda_{\geq r}:= v^{-1}([r, +\infty]) \quad \text{and} \quad \Lambda_{>r}:= v^{-1}((r, +\infty]), \quad \forall r\in (-\infty, +\infty),
$$
which are ideals of $\Lambda_{0}$. So $\Lambda_{0}$ is a short notation for $\Lambda_{\geq 0}$ and $\Lambda_{+}$ for $\Lambda_{>0}$. Later when we fix an energy bound $E>0$, we just write $\Lambda_{E}:= \Lambda_{0}/\Lambda_{\geq E}$.

\subsection{Hamiltonian Floer theory on closed manifolds}
Now we set up the background on Hamiltonian Floer theory. We work with a closed symplectic Calabi-Yau manifold $(M, \omega)$. Hence foundational details can be found in \cite{HS} and \cite{Salamon}.

A smooth function $H: M\rightarrow \mathbb{R}$ determines a smooth vector field $X_{H}$ such that $dH(\cdot)= \omega(X_{H}, \cdot)$. We say $H$ is a Hamiltonian function and $X_{H}$ is the associated Hamiltonian vector field. Let $\mathcal{L}M:= C^{\infty}(S^{1}, M)$ be the space of free loops in $M$, where we always view $S^{1}=\mathbb{R}/ \mathbb{Z}$ and write $t$ as the coordinate of $S^{1}$. We call $t$ the time variable.

Moreover, we can consider a family of functions $H_{t}: M\times S^{1}\rightarrow \mathbb{R}$ parameterized by $S^{1}$. Then we have time-dependent Hamiltonian vector fields $X_{H_{t}}$. Integrating it we obtain a family of Hamiltonian symplectic diffeomorphisms $\phi^{t}: M\rightarrow M$. A loop $\gamma \in \mathcal{L}M$ is called a time-one Hamiltonian orbit if $\gamma'(t)= X_{H_{t}}$. In this article, we only consider the component $\mathcal{L}_{0}M$ which contains contractible loops. Hence from now on, all Hamiltonian orbits are assumed to be contractible. We write
$$
\mathcal{CP}_{H_{t}}= \lbrace \gamma\in \mathcal{L}_{0}M\mid \gamma'(t)= X_{H_{t}}\rbrace
$$
as the set of contractible one-periodic orbits of $H_{t}$. An orbit is non-degenerate if the Poincar\'e return map
$$
d\phi^{1}: T_{\gamma(0)}M\rightarrow T_{\gamma(0)}M
$$
does not have eigenvalue one. And we say a Hamiltonian $H_{t}$ is non-degenerate if all of its one-periodic orbits are non-degenerate.

Next we assign an index $CZ(\gamma)$ to each orbit $\gamma$, the Conley-Zehnder index, see Lecture 2 in \cite{Salamon}. By the Calabi-Yau condition, this index does not depend on choices of cappings. We grade our orbits by setting
\begin{equation}\label{grading}
\mu_{H_{t}}(\gamma):= n+ CZ(\gamma)
\end{equation}
where $2n$ is the real dimension of $M$. (Our grading is different from that in \cite{Salamon}, since we use cohomology instead of homology.) When we say a degree-$k$ or an index-$k$ orbit we mean an orbit $\gamma$ with $\mu_{H_{t}}(\gamma)=k$. We remark that for a constant orbit of a $C^{2}$-small Morse function, its degree defined above equals its Morse index.

Therefore $\mathcal{CP}_{H_{t}}$ becomes a graded set
$$
\mathcal{CP}_{H_{t}}=\bigoplus_{k\in \mathbb{Z}} \mathcal{CP}_{H_{t}}^{k}
$$
where $\mathcal{CP}_{H_{t}}^{k}$ is the set of orbits with index $k$.

Then for a non-degenerate Hamiltonian $H_{t}$, we define
\begin{equation}\label{complex}
CF^{k}(H_{t}; \Lambda_{0})=\lbrace \sum_{i=1} c_{i}\gamma_{i}\mid c_{i}\in\Lambda_{0}, \gamma_{i}\in \mathcal{CP}_{H_{t}}^{k}\rbrace
\end{equation}
which is the free $\Lambda_{0}$-module generated by index-$k$ orbits. Similarly we can define
\begin{equation}\label{complex+}
CF^{k}(H_{t}; \Lambda)=\lbrace \sum_{i=1} c_{i}\gamma_{i}\mid c_{i}\in\Lambda, \gamma_{i}\in \mathcal{CP}_{H_{t}}^{k}\rbrace
\end{equation}
which is the $\Lambda$-vector space generated by index-$k$ orbits. We don't grade the formal variable $T$.

By using a family of compatible almost complex structures $J_{t}$, we can study the solutions of the Floer equation
\begin{equation}\label{Floer}
\partial_{s}u +J_{t}(\partial_{t}u -X_{H_{t}}) =0
\end{equation}
for $u: \mathbb{R}\times S^{1}\rightarrow M$. Here $s$ is the $\mathbb{R}$-coordinate and $t$ is the $S^{1}$-coordinate on the domain. For two orbits $\gamma_{-}, \gamma_{+}$ and a homotopy class $A\in \pi_{2}(M; \gamma_{-}\cup \gamma_{+})$, consider the solution space
\begin{equation}\label{moduli}
\begin{aligned}
\mathcal{M}(\gamma_{-}, \gamma_{+}; A)=& \lbrace u: \mathbb{R}\times S^{1}\rightarrow M\mid \partial_{s}u +J_{t}(\partial_{t}u -X_{H_{t}}) =0,\\
& u(-\infty, t)=\gamma_{-}, u(+\infty, t)=\gamma_{+}, [u]= A\in \pi_{2}(M; \gamma_{-}\cup \gamma_{+})\rbrace /\sim.
\end{aligned}
\end{equation}
There is an $\mathbb{R}$-action on this space by translating the $s$-coordinate of a solution $u$, and $\sim$ is the quotient of this $\mathbb{R}$-action. The $L^{2}$-energy of a solution $u$ is
$$
0\leq E(u)= \int \abs{\partial_{s}u}^{2}= \int u^{*}\omega +\int_{\gamma_{+}} H_{t} -\int_{\gamma_{-}} H_{t}.
$$
For generic pairs $(H_{t}, J_{t})$, the solution space is an $l$-dimensional manifold where
$$
\mu_{H_{t}}(\gamma_{+})-\mu_{H_{t}}(\gamma_{-})= l+1.
$$
We call a pair $(H_{t}, J_{t})$ satisfying the above condition a regular pair. By the Gromov-Floer compactness theorem, when $\gamma_{-}, \gamma_{+}, A$ are fixed, the above solution spaces admit compactifications by adding broken Floer trajectories and $J$-holomorphic sphere bubbles. The bubbles can be ruled out by using the Calabi-Yau condition when the moduli space is zero or one-dimensional, see \cite{HS}. There are also coherent orientations on these moduli spaces. In particular, when the moduli space is zero-dimensional, we can count the signed number of elements, which we denote by $n(\gamma_{-}, \gamma_{+}; A)$.

Then we define an operator
$$
d: CF^{k}(H_{t})\rightarrow CF^{k+1}(H_{t})
$$
with either $\Lambda_{0}$ or $\Lambda$ coefficients, by setting
\begin{equation}\label{operator}
d(\gamma_{-}):= \sum_{\gamma_{+}}\sum_{[u]=A} n(\gamma_{-}, \gamma_{+}; A)\cdot \gamma_{+} \cdot T^{\int u^{*}\omega +\int_{\gamma_{+}} H_{t} -\int_{\gamma_{-}} H_{t}}.
\end{equation}
The right-hand side is summed over all $\gamma_{+}$ with $\mu_{H_{t}}(\gamma_{+})-\mu_{H_{t}}(\gamma_{-})=1$ and all classes $A\in \pi_{2}(M; \gamma_{-}\cup \gamma_{-})$. It may not be a finite sum, but it converges as an element in $CF^{k}(H_{t})$, by the Gromov compactness theorem. Then we extend this operator $\Lambda_{0}$ or $\Lambda$-linearly to $CF^{k}(H_{t})$.

By the analysis of codimension one boundaries of $\mathcal{M}(\gamma_{-}, \gamma_{+}; A)$ with
$$
\mu_{H_{t}}(\gamma_{+})-\mu_{H_{t}}(\gamma_{-})=2,
$$
a big theorem in Hamiltonian Floer theory shows that $d^{2}=0$. Then we write the resulting cohomology groups as $HF^{k}(H_{t}, J_{t}; \Lambda_{0})$ and $HF^{k}(H_{t}, J_{t}; \Lambda)$.

Another theorem shows that $HF^{k}(H_{t}, J_{t}; \Lambda)$ is independent of the choices of generic pairs $(H_{t}, J_{t})$. Hence we can call it the Hamiltonian Floer cohomology of $M$. This invariance result is proved by considering continuation maps between different choices of $(H_{t}, J_{t})$. We sketch it here since we will use it later to define symplectic cohomology, see Section 3.4 in \cite{Salamon} for a full proof.

For simplicity we only vary $H_{t}$. The case for $J_{t}$ can be handled in the same way. Let $H^{\alpha}_{t}$ and $H^{\beta}_{t}$ be two non-degenerate Hamiltonians. Assume that both $(H^{\alpha}_{t}, J_{t})$ and $(H^{\beta}_{t}, J_{t})$ are regular for a fixed $J_{t}$. Then we choose a homotopy $H^{\alpha\beta}_{s, t}$ of Hamiltonians to connect $H^{\alpha}_{t}$ and $H^{\beta}_{t}$. That is,
\begin{equation}
\begin{aligned}
H^{\alpha\beta}_{s, t}: \mathbb{R}\times S^{1}\times M\rightarrow \mathbb{R};\\
H^{\alpha\beta}_{s, t}=
\begin{cases}
H^{\alpha}_{t}, \quad s\leq -1,\\
H^{\beta}_{t}, \quad s\geq 1.
\end{cases}
\end{aligned}
\end{equation}
Then we consider the $s$-dependent Floer equation
\begin{equation}\label{eq: continuation}
\partial_{s}u +J_{t}(\partial_{t}u -X_{H^{\alpha\beta}_{s,t}}) =0
\end{equation}
and the moduli space
$$
\begin{aligned}
\mathcal{M}(\gamma^{\alpha}_{-}, \gamma^{\beta}_{+}; A)=& \lbrace u: \mathbb{R}\times S^{1}\rightarrow M\mid \partial_{s}u +J_{t}(\partial_{t}u -X_{H^{\alpha\beta}_{s,t}}) =0,\\
& u(-\infty, t)=\gamma^{\alpha}_{-}, u(+\infty, t)=\gamma^{\beta}_{+}, [u]= A\in \pi_{2}(M; \gamma^{\alpha}_{-}\cup \gamma^{\beta}_{+})\rbrace.
\end{aligned}
$$
Note that now the equation is $s$-dependent hence there is no $\mathbb{R}$-action. For a generic path $H^{\alpha\beta}_{s, t}$, the above moduli space is a manifold of dimension $\mu_{H^{\beta}_{t}}(\gamma_{+})-\mu_{H^{\alpha}_{t}}(\gamma_{-})$. And it admits a similar compactification by adding broken trajectories. When $\mu_{H^{\beta}_{t}}(\gamma_{+})=\mu_{H^{\alpha}_{t}}(\gamma_{-})$, we define an operator
\begin{equation}\label{continuation}
f^{\alpha\beta}: CF^{k}(H^{\alpha}_{t})\rightarrow CF^{k}(H^{\beta}_{t})
\end{equation}
by setting
$$
f^{\alpha\beta}(\gamma^{\alpha}_{-}):= \sum_{\gamma^{\beta}_{+}}\sum_{[u]=A} n^{\alpha\beta}(\gamma^{\alpha}_{-}, \gamma^{\beta}_{+}; A)\cdot \gamma^{\beta}_{+} \cdot T^{\int u^{*}\omega +\int \partial_{s}(H^{\alpha\beta}_{s, t}(u(s, t)))},
$$
where $n^{\alpha\beta}(\gamma^{\alpha}_{-}, \gamma^{\beta}_{+}; A)$ is a signed count of elements of the above moduli space. Note that the weight
\begin{equation}\label{weight}
\int u^{*}\omega +\int \partial_{s}(H^{\alpha\beta}_{s, t}(u(s, t)))= \int\abs{\partial_{s}u}^{2} +\int \frac{\partial H^{\alpha\beta}_{s, t}}{\partial s}(u(s, t))
\end{equation}
is not necessarily non-negative, hence we need to use $\Lambda$-coefficients now. We call this weight the \textit{topological energy} of a Floer solution. (It is non-negative if the family of Hamiltonian functions satisfies that $\int \partial_{s}H^{\alpha\beta}_{s, t}\geq 0$.) Then one can show that $f^{\alpha\beta}$ is a chain map and moreover $f^{\alpha\beta}\circ f^{\beta\alpha}$ is chain homotopy equivalent to the identity map. 

We add one more lemma here which will be used frequently in later sections.

\begin{lemma}\label{translation}
Let $M$ be a Calabi-Yau manifold and $H_{t}$ be a non-degenerate Hamiltonian function on $M$. For a constant $\Delta >0$, suppose that we have a homotopy $\lbrace H^{s}_{t}\rbrace_{s\in [0,1]}$ such that $H^{s}_{t}= H_{t}+ s\Delta$. Then the continuation map between $CF(H^{0}_{t})$ and $CF(H^{1}_{t})$ is a multiplication by $T^{\Delta}$. In particular, when $\Delta=0$ this shows that the continuation map of a constant homotopy is the identity map.
\end{lemma}
\begin{proof}
Since our Hamiltonian functions are just a translation of a fixed one, the Floer equation of the continuation map (\ref{eq: continuation}) does not depend on $s$. Therefore any solution of it still carries an $\mathbb{R}$-action. By using the Calabi-Yau condition, we can pick a regular family $J_{t}$ for $H_{t}$. For two different orbits $\gamma$ and $\gamma'$ with index $k$, any Floer solution of the continuation map equation connecting $\gamma$ and $\gamma'$ carries an $\mathbb{R}$-action. So the corresponding moduli space is at least one-dimensional, contradicting that both $\gamma$ and $\gamma'$ have index $k$.

Then the continuation maps only exist when $\gamma=\gamma'$ and it will be a constant map. One can directly check that the constant map is regular and has contribution one. Hence the continuation map between $CF(H^{0}_{t})$ and $CF(H^{1}_{t})$ is an identity matrix, weighted by the change of the Hamiltonian function which is $T^{\Delta}$.
\end{proof}

\subsection{Liouville domain and contact cylinder}
Let $(C, \alpha)$ be a contact manifold with a contact form $\alpha$. The \textit{Reeb vector field} of $\alpha$ is the unique vector field $R_{\alpha}$ on $C$ such that
$$
d\alpha(R_{\alpha}, \cdot)=0, \quad \alpha (R_{\alpha})=1.
$$
Then a \textit{Reeb orbit of length $\lambda >0$} is a map
$$
\gamma(t): \mathbb{R}/\lambda\mathbb{Z}\rightarrow C, \quad \frac{d}{dt}\gamma(t)=R_{\alpha}.
$$
We write $\Gamma_{\alpha, \lambda}\subset C$ the set formed by Reeb orbits of length $\lambda$ and $\Gamma_{\alpha}:= \cup_{\lambda >0}\Gamma_{\alpha, \lambda}$. We say a Reeb orbit is non-degenerate if the Poincar\'e return map of the Reeb flow does not have eigenvalue one. And we say a contact form $\alpha$ is non-degenerate if all of its orbits are non-degenerate. We say $\alpha$ is Morse-Bott non-degenerate if for all $\lambda>0$, the set $\Gamma_{\alpha, \lambda}$ is a closed submanifold in $C$, the rank of $d\alpha\mid_{\Gamma_{\alpha, \lambda}}$ is locally constant, and $T_{p}\Gamma_{\alpha, \lambda}=\ker (T_{p}\phi_{\lambda}-id)$ for all $p\in \Gamma_{\alpha, \lambda}$, where $\phi_{t}$ is the Reeb flow. For a Morse-Bott non-degenerate contact form, one can define a Conley-Zehnder index of its Reeb orbits, see Section 2.1 in \cite{CM} and references therein for more details on Reeb orbits and their indices. Then we have the following definition.

\begin{definition}\label{index}(Definition 1.12 \cite{TVar})
	Suppose $(M, \omega)$ is Calabi-Yau. A contact hypersurface $(C, \alpha)$ in $M$ is called
	index bounded if 
	\begin{enumerate}
		\item $\alpha$ is a non-degenerate contact form.
		\item All of its Reeb orbits are contractible inside $M$. 
		\item For any integer $k$, the lengths of the Reeb orbits of Conley-Zehnder index $k$ are bounded from above. 
	\end{enumerate}
	Similarly, we call a Liouville domain $(D, \theta)$ in $M$ index bounded if its boundary $(\partial D, \alpha=\theta\mid_{\partial D})$ is index bounded.
\end{definition}

Let $(D, \theta)$ be a Liouville domain in $(M, \omega)$. For small $\epsilon>0$ there is an embedding
$$
\bar{C}=[1-\epsilon, 1+\epsilon]\times \partial D\rightarrow M
$$
such that $\partial D$ is the image of $\lbrace 1\rbrace\times \partial D$. We do not distinguish $\bar{C}$ from its image. Then $\bar{C}$ is called a contact cylinder associated to the Liouville domain $D$ if $\omega\mid_{\bar{C}}=d(r\alpha)$ where $r$ is the coordinate on $[1-\epsilon, 1+\epsilon]$. And we will write $D_{1+\epsilon}=D\cup \bar{C}$ as a compact neighborhood of $D$. The one-form $\theta$ on $D$ smoothly extends to a one-form $\theta$ on $D_{1+\epsilon}$ such that $\omega\mid_{D_{1+\epsilon}}=d\theta$.

\subsection{Relative symplectic cohomology of a Liouville domain}
The relative symplectic cohomology is the homology of a suitable limit of complexes derived from Hamiltonian Floer theory of $M$. The whole construction \cite{Var} of relative symplectic cohomology defines a module $SH_{M}(D)$ over the universal Novikov ring $\Lambda_{0}$, for any compact subset $D$ of a closed symplectic manifold $M$. Now we briefly review its definition when $D$ is a convex domain in a symplectic Calabi-Yau manifold $M$.

The following data is called an \textit{acceleration data} for $D$:
\begin{enumerate}
\item $H_{1,t}\leq H_{2,t}\leq \cdots$ a monotone sequence of non-degenerate Hamiltonian functions, such that $H_{i,t}(x)\to 0$ on $D$ and $H_{i,t}(x)\to +\infty$ on $M-D$. 

\item Monotone homotopies of Hamiltonians $\lbrace H_{s,t}\rbrace_{s\in [i,i+1]}$ for all $i$, which mean that $H_{s,t}(x)\geq H_{s',t}(x)$ if $s\geq s'$ and $H_{s,t}=H_{i,t}$ if $s=i$.

\item A family of almost complex structures $\lbrace J_{s,t}\rbrace_{(s,t)\in [1,+\infty)\times S^{1}}$ such that for each $i$, $(H_{i,t}, J_{i,t})$ is a regular pair, and for each $i$, $(H_{s,t}, J_{s,t})_{s\in [i,i+1]}$ is a regular homotopy.
\end{enumerate}
From an acceleration data, we obtain a sequence of chain complexes over $\Lambda_{0}$:
$$
\mathcal{CF}^{k}=CF^{k}(H_{1,t})\rightarrow CF^{k}(H_{2,t})\rightarrow \cdots
$$
connected by continuation maps. Here each $CF^{k}(H_{i,t})$ is the degree-$k$ Floer complex of the Hamiltonian $H_{i,t}$. Since $H_{i,t}\leq H_{i+1,t}$ are connected by a monotone family of Hamiltonians, the weight (\ref{weight}) in the continuation map is non-negative.

Then the relative symplectic cohomology module $SH^{k}_{M}(D; \Lambda_{0})$ is defined as the cohomology
\begin{equation}\label{def1}
H(\widehat{tel}(\mathcal{CF}^{k}); \Lambda_{0})
\end{equation}
of the completion $\widehat{tel}(\mathcal{CF}^{k})$ of the telescope
$$
tel(\mathcal{CF}^{k})=\bigoplus_{n\in \mathbb{Z}_{+}}(C_{n}[1]\oplus C_{n}).
$$
Here we write $C_{n}=CF^{k}(H_{n,t})$. Another algebraic way to define it is that
\begin{equation}\label{def2}
\begin{aligned}
& SC^{k}_{M}(D; \Lambda_{0}):= \varprojlim_{r\geq 0} \varinjlim_{n}C_{n}\otimes_{\Lambda_{0}}\Lambda_{0}/\Lambda_{\geq r}\\
& SH^{k}_{M}(D; \Lambda_{0}):= H(SC^{k}_{M}(D); \Lambda_{0}).
\end{aligned}
\end{equation}
For the equivalence of these two definitions, see Section 2 in \cite{Var}. And it is also shown that the definition of $SH^{k}_{M}(D; \Lambda_{0})$ is independent of various choices.

\begin{proposition}(Proposition 3.3.4, \cite{Var})
\begin{enumerate}
\item Let $H_{s}$ and $H_{s}'$ be two different acceleration data, then $H(SC_{M}(D, H_{s}))\cong H(SC_{M}(D, H_{s}'))$ canonically. Therefore we simply denote this invariant by $SH_{M}(D)$.

\item Let $\phi: M\rightarrow M$ be a symplectomorphism. There exists a canonical isomorphism $SH_{M}(D)=SH_{M}(\phi(D))$ by relabeling an acceleration data by the map $\phi$.

\item For $D\subset D'$, there exists canonical restriction maps $SH_{M}(D')\rightarrow SH_{M}(D)$. This satisfies the presheaf property.
\end{enumerate}
\end{proposition}

Hence we can just write this invariant as $SH_{M}(D):=SH_{M}(D; \Lambda_{0})$ and its torsion-free part $SH_{M}(D; \Lambda_{0})\otimes_{\Lambda_{0}}\Lambda$. This invariant has many good properties. Notably it satisfies the Mayer-Vietoris exact sequence in some settings. Another property we will keep using here is the \textit{stably displaceability condition}.

\begin{theorem}(Theorem 4.0.1 and Remark 4.3.1 \cite{Var})\label{t: stable}
If the compact subset $D\subset M$ is stably displaceable then $SH_{M}(D; \Lambda_{0})\otimes_{\Lambda_{0}}\Lambda =0$.
\end{theorem}

In practice when $D$ is a convex domain, with a non-degenerate contact form on its boundary, we will use a particular class of acceleration data to compute the relative symplectic cohomology. For small $\epsilon>0$ we fix a contact cylinder $\bar{C}$ associated to $D$ and write $D_{1+\epsilon}=D\cup \bar{C}$ as a compact neighborhood of $D$. We introduce the notion of $S$-shape Hamiltonian functions, see Figure 1.

\begin{figure}
  \begin{tikzpicture}[xscale=0.8, yscale=0.8]
  \draw [->] (-2,0)--(4.5,0);
  \node [right] at (4.5,0) {$r$};
  \draw (-1,-0.2) to [out=0, in=250] (1,1);
  \draw (1,1)--(2,4);
  \draw (2,4) to [out=70, in=180] (3, 5);
  \draw (-1,0.1)--(-1,-0.1);
  \node [above] at (-1,0.1) {\small{$1-\frac{\epsilon}{4}$}};
  \draw (0,0.1)--(0,-0.1);
  \node [above] at (0,0.1) {\small{$1$}};
  \draw (1,0.1)--(1,-0.1);
  \node [below] at (1,-0.1) {\small{$1+\frac{\epsilon}{4}$}};
  \draw (2,0.1)--(2,-0.1);
  \node [below] at (2,-0.1) {\small{$1+\frac{\epsilon}{2}$}};
  \draw (3,0.1)--(3,-0.1);
  \node [below] at (3,-0.1) {\small{$1+\frac{3\epsilon}{4}$}};
  \draw (4,0.1)--(4,-0.1);
  \node [below] at (4,-0.1) {\small{$1+\epsilon$}};
  \draw [snake it] (-2,-0.2)--(-1,-0.2);
  \draw [snake it] (3,5)--(4,5);
  \end{tikzpicture}
  \caption{Hamiltonian functions in the cylindrical coordinate.}\label{Ham}
\end{figure}
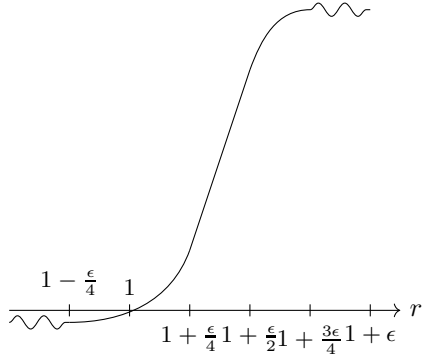

\begin{definition}
A time-independent Hamiltonian function $H: M\rightarrow \mathbb{R}$ is called an $S$-shape Hamiltonian if
\begin{enumerate}
\item $H$ is cylindrical on the region $[1-\frac{\epsilon}{4}, 1+\frac{3\epsilon}{4}]\times \partial D\subset \bar{C}$. That is, $H(x)= H(x')$ if $r(x)=r(x')$ where $r$ is the cylindrical coordinate. So we can write $H(x)=h(r(x))$ for some $h: [1-\frac{\epsilon}{4}, 1+\frac{3\epsilon}{4}]\rightarrow \mathbb{R}$ on the cylinder region;
\item $h'(r)$ is concave and $h(r)=\lambda r+ m$ on $[1+\frac{\epsilon}{4}, 1+\frac{\epsilon}{2}]\times \partial D\subset \bar{C}$ for some constants $\lambda>0$ and $m$;
\item the linear slope $\lambda$ is not in the action spectrum of the contact form;
\item $H$ is a $C^{2}$-small Morse function on $D_{1-\frac{\epsilon}{4}}$, and it is a Morse function on $M-D_{1+\frac{3\epsilon}{4}}$ with small derivatives, such that it only has constant orbits outside $[1-\frac{\epsilon}{4}, 1+\frac{3\epsilon}{4}]\times \partial D$.
\end{enumerate}
\end{definition}

By this special shape of our Hamiltonian functions, we have further description of their one-periodic orbits. First in the region where the Hamiltonians has small derivatives, there are only constant orbits. Hence all non-constant orbits lie in the contact cylinder $\bar{C}$. By a direct computation we have that
$$
X_{H_{n,t}}(r_{0}, c_{0})= -\partial_{r}H_{n,t}(r_{0})\cdot R_{\alpha}(c_{0}), \quad \forall (r_{0}, c_{0})\in \bar{C}=[1-\epsilon, 1+\epsilon]\times \partial D.
$$
That is, the Hamiltonian vector field is proportional to the Reeb vector field with the negative slope as the ratio. Hence any $\tau$-periodic Reeb orbit gives rise to a one-periodic Hamiltonian orbit in $r_{0}\times \partial D$ if and only if
$$
\tau=\abs{\int_{0}^{1} \partial_{r}H_{n,t}(r_{0})dt}.
$$
Since the Hamiltonian is linear in the middle of the cylindrical region, with a slope which is not in the action spectrum of the contact form, there are no orbits in this region. Hence all the non-constant one-periodic Hamiltonian orbits can be separated into two groups. One group is located in the region $[1-\frac{\epsilon}{4}, 1+\frac{\epsilon}{4}]\times \partial D$ and we call them \textit{lower} orbits. The other group is located in the region $[1+\frac{\epsilon}{2}, 1+\frac{3\epsilon}{4}]\times \partial D$ and we call them \textit{upper} orbits.

Previously, an $S$-shape Hamiltonian function is time-independent. So the non-constant orbits appear in $S^{1}$-families. Now we use small time-dependent perturbations to make them non-degenerate, by using the technique in \cite{CFHW}.

\begin{proposition}\label{p: perturbation}
(Lemma 2.1 and Proposition 2.2 in \cite{CFHW}) Let $H$ be a time-independent $S$-shape Hamiltonian function and let $\gamma$ be a non-constant one-periodic orbit of $H$ such that $\gamma$ is transversally non-degenerate. Pick $U$ to be a neighborhood of $\gamma$ which does not contain other one-periodic orbits, then there exists a time-dependent function $H_{t}$ such that
\begin{enumerate}
\item the support of $H- H_{t}$ is in $U$;
\item there are exactly two one-periodic orbits $\gamma_{\pm}$ of $H_{t}$ in $U$;
\item $\int \gamma^{*}\theta= \int \gamma_{+}^{*}\theta= \int \gamma_{-}^{*}\theta$, where $\theta$ is the Liouville one-form in the cylindrical region;
\item the difference between the Conley-Zehnder indices of $\gamma$ and $\gamma_{\pm}$ is bounded by one.
\end{enumerate}
\end{proposition}

Note that the orbit $\gamma$ is both a Hamiltonian orbit of $H$ and a Reeb orbit of the contact form. One can compute its index in two ways. The following lemma relates these two indices.

\begin{lemma}\label{comparison}(Lemma 5.25, \cite{M2020})
Let $\bar{C}=[1-\epsilon, 1+\epsilon]\times C$ be an index bounded contact cylinder with cylindrical coordinate $r$ and associated contact form $\alpha$ and let $\pi :\bar{C}\rightarrow C$ be the natural projection map. Let $f: [1-\epsilon, 1+\epsilon]\rightarrow \mathbb{R}$ be a smooth function, $R_{\lambda, [-m,m]}$ be the set of Reeb orbits of length $\lambda$ and index in $[-m, m]$ and $O_{\lambda, [-m,m]}$ be the set of one-periodic orbits of $f(r)$ contained
in $\lbrace r=(f')^{-1}(\lambda)\rbrace$ of index in $[-m, m]$ which are null homologous in M. Then the map
$$
O_{\lambda, [-m,m]}\rightarrow R_{\lambda, [-m-\frac{1}{2},m+\frac{1}{2}]}
$$
sending $\gamma: \mathbb{R}/\mathbb{Z}\rightarrow \bar{C}$ to $\pi \circ \gamma \circ b_{\lambda}$ is well-defined where
$$
b_{\lambda}: [0,\lambda]\rightarrow [0,1], \quad b_{\lambda}(t):= t/\lambda, \quad \forall t\in [0,\lambda].
$$
\end{lemma}

Hence we can use the index bounded condition, which was previously defined for Reeb orbits, to the setting of Hamiltonian orbits. Let $D$ be a convex domain with an index bounded boundary in a Calabi-Yau manifold. We start with a time-independent $S$-shape Hamiltonian function and perturb it. Before perturbation, a non-constant orbit $\gamma$ satisfies the index-length relation in Definition \ref{index} since it comes from a Reeb orbit. After perturbation, by the above proposition and lemma, the index-length relation still holds for new orbits $\gamma_{\pm}$.

Now we say a time-dependent function $H_{t}$ is a time-dependent $S$-shape Hamiltonian function if it is a perturbation of a time-independent $S$-shape Hamiltonian function as in Proposition \ref{p: perturbation}. Note that we only perturb the region where non-constant orbits lie, so we can still talk about upper and lower orbits after the perturbation.

\begin{definition}
Let $D$ be a convex domain in a Calabi-Yau manifold $M$. A time-dependent $S$-shape Hamiltonian function $H_{t}$ is called index bounded if for any integer $k$, there exists a constant $\mu_{k}>0$ such that $\int \gamma^{*}\theta<\mu_{k}$ for all degree-$k$ one-periodic orbits of $H_{t}$.
\end{definition}

The above discussion says that if $D$ is a convex domain with an index bounded boundary, then we can always find time-dependent non-degenerate $S$-shape Hamiltonian functions which are index bounded. In practice we will use families of time-dependent non-degenerate $S$-shape Hamiltonian functions to compute $SH_{M}(D)$.

\subsection{Hamiltonian Floer theory on manifolds with convex boundary}
Now we review the construction of Hamiltonian Floer theory on convex manifolds, and fix our notation along the way. Symplectic cohomology was first introduced by Cieliebak-Floer-Hofer \cite{CFH} and Viterbo \cite{Vit2} in the exact setting, and by Ritter \cite{R2010} in the non-exact setting.

Let $(M, \omega= d\theta)$ be a Liouville domain, and let $\alpha :=\theta\mid_{\partial M}$ be the contact form. Then we can attach a cylindrical end to $M$ to get an open symplectic manifold
$$
\hat{M}:= M \cup (\partial M\times [1, +\infty)).
$$
Let $r$ be the coordinate on $[1, +\infty)$. We equip the manifold $\hat{M}$ with a smooth symplectic form $\hat{\omega}$, where $\hat{\omega}= \omega$ on $M$ and $\hat{\omega}=d(r\alpha)$ on $\partial M\times [1, +\infty)$. And $(\hat{M}, \hat{\omega})$ will be called the \textit{completion} of $M$. In the following we assume that $\alpha$ on $\partial M$ is non-degenerate.

Now we define \textit{admissible} Hamiltonian functions we will use. A Hamiltonian function $H_{t}: S^{1}\times \hat{M}\rightarrow \mathbb{R}$ is called admissible if
\begin{enumerate}
\item it is a negative time-independent Morse function on $M$.
\item all its contractible one-periodic orbits in $\hat{M}$ are non-degenerate.
\item it is a linear function just depending on $r$ with a positive slope on $\partial M\times [R_{0}, +\infty)$ for some $R_{0}>1$.
\item the slope of the linear part is not an element in $Spec(\alpha)$.
\end{enumerate}

For an admissible Hamiltonian function, there are only finitely many one-periodic orbits. Next with an admissible Hamiltonian $H_{t}$, we consider the degree-$k$ Floer complex $CF^{k}(H_{t})$ as in (\ref{complex}) and (\ref{complex+}). Then for suitably chosen almost complex structures, we can use moduli spaces of Floer solutions to define differentials and continuation maps as in the closed case.

For a monotone family $H_{i,t}$ such that $H_{i,t}\leq H_{i+1,t}$ and the linear slope of $H_{i,t}$ goes to positive infinity, we have a sequence of complexes
$$
\mathcal{CF}^{k}=CF^{k}(H_{1,t})\rightarrow CF^{k}(H_{2,t})\rightarrow \cdots
$$
connected by continuation maps. Here all Floer differentials and continuation maps are weighted by the topological energy, see (\ref{weight}). Since the symplectic form on our Liouville domain is exact, the images of an orbit under Floer differentials and continuation maps are finite sums of other orbits. Hence we can both define the classical symplectic cohomology over $\mathbb{C}$ or over $\Lambda_{0}$. The former theory can be defined by the later one by setting $T=1$.

The classical symplectic cohomology of $M$ over $\Lambda_{0}$ is defined as
$$
SH^{k}(M; \Lambda_{0}):= H(tel(\mathcal{CF}^{k})).
$$
An essential difference between this definition and that of the relative symplectic cohomology is that the classical one does not complete $tel(\mathcal{CF}^{k})$ before taking homology. 

The classical symplectic cohomology of $M$ over $\mathbb{C}$ is defined as
$$
SH^{k}(M; \mathbb{C}):= H(tel(\mathcal{CF}^{k})\mid_{T=1}).
$$
In other words, all differentials and continuation maps are defined without weights. Since $(M, d\theta)$ is an exact symplectic manifold, this reduction to $T=1$ is well-defined.

\subsection{Lower semi-continuous Hamiltonian functions}
Defining the classical $SH^{k}(M; \mathbb{C})$ via taking a direct limit is equivalent to the definition via a single Hamiltonian function that are quadratic at infinity, see Section 3 in \cite{Se1}. Similarly, we will see that the relative symplectic cohomology $SH_{M}(D)$ is related to a kind of Hamiltonian Floer theory for a lower semi-continuous Hamiltonian function $F^{D}_{0}$, where $F^{D}_{0}\mid_{D}=0$ and $F^{D}_{0}\mid_{M-D}=+\infty$. A general study of Hamiltonian Floer theory of lower semi-continuous functions can be found in \cite{Gro} and \cite{M2020}. Now we will discuss a special case of it for our purpose.

Fix a closed symplectic Calabi-Yau manifold $M$ and let $F: M\times S^{1}\rightarrow \mathbb{R}$ be a lower semi-continuous function. Pick a monotone sequence $\lbrace H_{n,t}\rbrace$ of non-degenerate Hamiltonian functions such that
$$
H_{1,t}\leq H_{2,t}\leq \cdots \leq H_{n,t} \leq \cdots \rightarrow F.
$$
The Hamiltonian Floer cohomology $HF(F)$ is defined as
\begin{equation}
H(\widehat{tel}(CF(H_{1,t})\rightarrow CF(H_{2,t})\rightarrow \cdots)).
\end{equation}
Also we have an equivalent definition in terms of (\ref{def2}).

From its definition, we can see that the lower semi-continuous Hamiltonian Floer cohomology $HF(F)$ is a generalization of the relative symplectic cohomology. For a domain $D$ in $M$, we have that
$$
SH_{M}(D)\cong HF(F^{D}_{0}).
$$

This indicator-type function $F^{D}_{0}$ is a very degenerate one, since it is identically zero on $D$. In practice, it is often more handy to replace its part on $D$ by a fixed non-degenerate Hamiltonian function $f$. And our main result of this subsection is the following isomorphism, which shows that the choice $f$ does not matter if we work over the Novikov field.

\begin{proposition}\label{single}
Let $D$ be a domain in $M$. Let $F: M\times S^{1}\rightarrow \mathbb{R}$ be a lower semi-continuous function such that $F$ is a smooth negative non-degenerate Hamiltonian function $f$ on $D\times S^{1}$ and $F=+\infty$ otherwise. Then we have that
$$
SH_{M}(D)\otimes_{\Lambda_{0}}\Lambda \cong HF(F)\otimes_{\Lambda_{0}}\Lambda.
$$
\end{proposition}
\begin{proof}
Pick a cofinal family $\lbrace H_{n,t}\rbrace$ of non-degenerate Hamiltonian functions such that
$$
H_{1,t}\leq H_{2,t}\leq \cdots \leq H_{n,t} \leq \cdots \rightarrow F^{D}_{0}
$$
and a cofinal family $\lbrace K_{n,t}\rbrace$ of non-degenerate Hamiltonian functions such that
$$
K_{1,t}\leq K_{2,t}\leq \cdots \leq K_{n,t} \leq \cdots \rightarrow F.
$$
By the non-degeneracy of $f$ in $D\times S^{1}$, we can choose $K_{n,t}$ such that $K_{n,t}(x)=f(x,t)$ for all $n\in\mathbb{N}, (x,t)\in D\times S^{1}$. Moreover, since $F$ is negative on $D$ we can choose above families such that $H_{n,t}\geq K_{n,t}$ for all $n$. Moreover, we can assume that
$$
\max_{(x,t)\in M\times S^{1}}(H_{n,t}- K_{n,t})= \max_{(x,t)\in D\times S^{1}}(H_{n,t}- K_{n,t}), \quad \forall n\in\mathbb{N}.
$$
Consider the continuation map
$$
CF(K_{i,t})\rightarrow CF(H_{i,t})
$$
over $\Lambda_{0}$ for each $i$. These continuation maps induce a $\Lambda_{0}$-module map
$$
HF(F)\rightarrow HF(F^{D}_{0})\cong SH_{M}(D).
$$

Next we set $\Delta:= -\min_{(x,t)\in D\times S^{1}} K_{1,t}$, and set $F+\Delta$ to be the lower semi-continuous function such that it is $f+\Delta$ on $D\times S^{1}$ and positive infinity otherwise. Note that the family $\lbrace K_{n,t}+ \Delta\rbrace$ is a cofinal family for the function $F+\Delta$, so on $D\times S^{1}$ we have that $K_{n,t}+\Delta \geq K_{1,t}+\Delta \geq 0\geq H_{n,t}$ for all $n$. On the other hand, since $H_{n,t}- K_{n,t}$ takes maximum on $D\times S^{1}$, we have $K_{n,t}+\Delta \geq H_{n,t}$ globally on $M\times S^{1}$, see Figure 2. This gives the second $\Lambda_{0}$-module map in the following
$$
HF(F)\rightarrow HF(F^{D}_{0}) \rightarrow HF(F+\Delta) \rightarrow HF(F^{D}_{0}+\Delta').
$$
Similarly we can find some other constant $\Delta'$ to define the third map. (For example, we can take $\Delta'= \Delta$. What we need is that $H_{n,t}+\Delta' \geq K_{n,t}+\Delta$ for all $n$.) Since the data to define $HF(F+\Delta)$ is just a translation of the data to define $HF(F)$, the composition of the first two maps is the multiplication by $T^{\Delta}$, by Lemma \ref{translation}. By the same reason, the composition of the last two maps is the multiplication by $T^{\Delta'}$. After tensoring with the Novikov field, these compositions become isomorphisms, which show that the three $\Lambda_{0}$-modules $HF(F),HF(F^{D}_{0})$ and $HF(F+\Delta)$ are all isomorphic over $\Lambda$.
\end{proof}

\begin{figure}
\begin{tikzpicture}[xscale=0.6, yscale=0.6]
  \draw [->] (0,0)--(12,0);
  \draw [snake it] (0,-1)--(4,-1);
  \draw (4,-1) to [out=0, in=225] (5,0);
  \draw (5,0)--(7,3);
  \draw (7,3) to [out=45, in=180] (8,4);
  \draw [snake it] (8,4)--(10,4);
  \node at (11, 4) {$K_{n,t}$};

  \draw [snake it] (0,1)--(4,1);
  \draw (4,1) to [out=0, in=225] (5,2);
  \draw (5,2)--(7,5);
  \draw (7,5) to [out=45, in=180] (8,6);
  \draw [snake it] (8,6)--(10,6);
  \node at (11.5, 6) {$K_{n,t}+\Delta$};

  \draw [blue, snake it] (0,-0.5)--(4,-0.5);
  \draw [blue] (4,-0.5) to [out=0, in=225] (5,0.5);
  \draw [blue] (5,0.5)--(7,3.5);
  \draw [blue] (7,3.5) to [out=45, in=180] (8,4.5);
  \draw [blue, snake it] (8,4.5)--(10,4.5);
  \node [blue] at (11, 4.5) {$H_{n,t}$};

\end{tikzpicture}
\caption{A sandwich of Hamiltonian functions}
\end{figure}
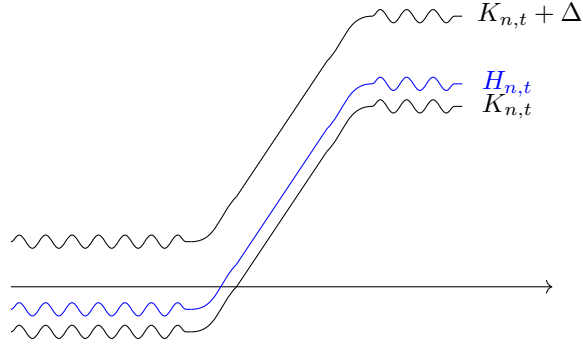

Therefore if we only care about the relative symplectic cohomology over the Novikov field, we can relax the condition of the acceleration data of Hamiltonian functions we used for computations. That is, the Hamiltonian functions converge to a fixed negative Morse function on $D$, instead of converging to zero.

\section{A filtration on the completed telescope}\label{sec:fil}
Let $(M, \omega)$ be a closed Calabi-Yau symplectic manifold and let $D\subset M$ be a Liouville domain. Recall that the relative symplectic cohomology $SH_{M}(D)$ is the homology of the completion $\widehat{tel}(\mathcal{CF})$ of the telescope
$$
tel(\mathcal{CF})=\bigoplus_{i\in \mathbb{N}}(C_{i}[1]\oplus C_{i}).
$$
Now we will define a filtration on $\widehat{tel}(\mathcal{CF})$ which is compatible with the differentials.

\subsection{Auxiliary symplectic forms}
First we define some auxiliary symplectic forms with respect to our Liouville domain $D$. The boundary $\partial D$ is a contact hypersurface in $M$ and for small $\epsilon >0$ we write $D_{1+\epsilon}=D\cup ([1, 1+\epsilon]\times\partial D )$ as a compact neighborhood of $D$. That is, we fix a choice of a contact cylinder associated to $D$. On $D_{1+\epsilon}$ there is a 1-form $\theta$ such that $d\theta =\omega$. Then we can interpolate $\omega$ from $D_{1+\frac{3}{4}\epsilon}$ to $M-D_{1+\epsilon}$.

\begin{lemma}\label{aux}
There exists a global 2-form $\tilde{\omega}$ and a 1-form $\tilde{\theta}$ on $M$ such that
\begin{enumerate}
\item $\omega=\tilde{\omega}+ d\tilde{\theta}$ on $M$;
\item The support of $\tilde{\theta}$ is in $D_{1+\epsilon}$ and $\tilde{\theta}=\theta$ in $D_{1+\frac{3}{4}\epsilon}$;
\item The support of $\tilde{\omega}$ is in $M-D_{1+\frac{3}{4}\epsilon}$ and $\tilde{\omega}=\omega$ in $M-D_{1+\epsilon}$.
\end{enumerate}
\end{lemma}
\begin{proof}
Let $\rho(r): [1,2]\rightarrow \mathbb{R}$ be a smooth increasing function such that
$$
\rho(r)= \left\lbrace
\begin{aligned}
& 0, \quad 1\leq r\leq 1+\frac{3}{4}\epsilon;\\
& 1, \quad 1+\epsilon \leq r.
\end{aligned}
\right.
$$
Next we define that
$$
\tilde{\omega}\mid_{x}= \left\lbrace
\begin{aligned}
& 0, \quad x\in D_{1+\frac{3}{4}\epsilon};\\
& d(\rho(r)\theta), \quad x\in D_{1+\epsilon}-D_{1+\frac{3}{4}\epsilon};\\
& \omega, \quad x\in M-D_{1+\epsilon}
\end{aligned}
\right.
$$
and
$$
\tilde{\theta}\mid_{x}= \left\lbrace
\begin{aligned}
& \theta, \quad x\in D_{1+\frac{3}{4}\epsilon};\\
& (1-\rho(r))\theta, \quad x\in D_{1+\epsilon}-D_{1+\frac{3}{4}\epsilon};\\
& 0, \quad x\in M-D_{1+\epsilon}
\end{aligned}
\right.
$$
Then we can check that $\tilde{\omega}$ and $\tilde{\theta}$ satisfy the conditions we need, see Figure 3.
\end{proof}

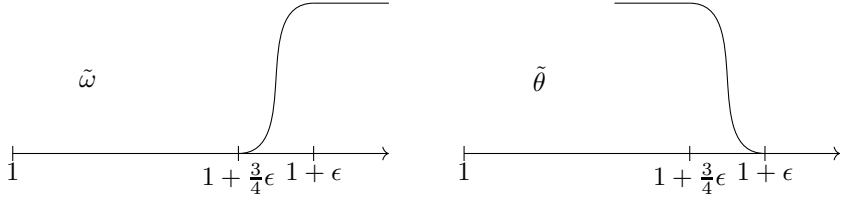
\begin{figure}
  \begin{tikzpicture}
  \node at (2,1) {$\tilde{\omega}$};
  \draw [->] (1,0)--(6,0);
  \draw (4,0) to [out=0, in=180] (5,2);
  \draw (5,2)--(6,2);
  \draw (1,0.1)--(1,-0.1);
  \node [below] at (1,0) {$1$};
  \draw (4,0.1)--(4,-0.1);
  \node [below] at (4,0) {$1+\frac{3}{4}\epsilon$};
  \draw (5,0.1)--(5,-0.1);
  \node [below] at (5,0) {$1+\epsilon$};
  \node at (8,1) {$\tilde{\theta}$};
  \draw [->] (7,0)--(12,0);
  \draw (10,2) to [out=0, in=180] (11,0);
  \draw (9,2)--(10,2);
  \draw (7,0.1)--(7,-0.1);
  \node [below] at (7,0) {$1$};
  \draw (10,0.1)--(10,-0.1);
  \node [below] at (10,0) {$1+\frac{3}{4}\epsilon$};
  \draw (11,0.1)--(11,-0.1);
  \node [below] at (11,0) {$1+\epsilon$};
  \end{tikzpicture}
  \caption{Cut-off functions for $\tilde{\omega}$ and $\tilde{\theta}$.}
\end{figure}

By the definition of $\tilde{\omega}$, it represents a cohomology class $[\tilde{\omega}]\in H^{2}(M,D)$. The following exact sequence of de Rham cohomology
$$
\cdots \rightarrow H^{1}(D)\rightarrow H^{2}(M,D)\xrightarrow{j} H^{2}(M)\rightarrow H^{2}(D)\rightarrow \cdots
$$
says that $j([\tilde{\omega}])=[\omega]$. So we also call $[\tilde{\omega}]$ a lift of $[\omega]$. In the case that $[\tilde{\omega}]$ takes integral values on $H_{2}(M, D; \mathbb{Z})$, we say $[\tilde{\omega}]$ is an integral lift of $[\omega]$. From now on, our Liouville domain $D$ is always equipped with such an auxiliary form $\tilde{\omega}$ and we assume it represents an integral class.

\begin{example}
	If $D$ is simply-connected, then $H^{1}(D)=0$ and we have a unique lift $[\tilde{\omega}]$ of $[\omega]$. Moreover, the map $H_{2}(M)\to H_{2}(M, D)$ is surjective. Hence if $[\omega]$ represents an integral class in $H^{2}(M)$ then $[\tilde{\omega}]$ is automatically integral. Similar conclusion holds when the map $\pi_{1}(D)\to \pi_{1}(M)$ is injective.
\end{example}

For a Hamiltonian vector field which is small on $M-D_{1+\frac{3\epsilon}{4}}$, the Floer equation becomes very close to the genuine Cauchy-Riemann equation on that region. This implies the positivity of the $\omega$-energy of solutions on that region. To prove this first we recall the following computation for a solution to the Floer equation. Let $H$ be a Hamiltonian function and $J$ be a compatible almost complex structure. Let $g(\cdot, \cdot):= \omega(\cdot, J\cdot)$ be the induced Riemannian metric. For a solution $u$ with finite energy we have that
$$
\omega([u])= \int u^{*}\omega= \int\omega(\partial_{s}u, \partial_{t}u)=\int \abs{\partial_{s}u}_{g}^{2}+ \int \omega(\partial_{s}u, X_{H}).
$$
For the second term we have that
$$
\begin{aligned}
&\int \omega(\partial_{s}u, X_{H})= -\int dH(\partial_{s}u)\\
&=-\int_{-\infty}^{+\infty} \frac{\partial}{\partial s}\int_{0}^{1}H(u) dtds= -(\int_{0}^{1}H(u(+\infty, t)) dt- \int_{0}^{1}H(u(-\infty, t)) dt)\\
&\geq -\norm{H}
\end{aligned}
$$
where $\norm{H}= \int_{0}^{1} (\max_{x\in M}H(x, t)- \min_{x\in M}H(x, t))dt$ is the Hofer norm of $H$. When we restrict our function to some region of $M$, the relative Hofer norm is defined in a similar way by taking max and min on that region.

\begin{lemma}\label{pos}
Let $(M, \omega)$ be a closed symplectic manifold and $D$ be a Liouville domain in $M$. Let $\tilde{\omega}$ be an auxiliary form constructed above. We assume that $[\tilde{\omega}]$ is an integral lift of $[\omega]$. Then for any non-degenerate Hamiltonian function $H$ with $\norm{H}_{M-D_{1+\frac{3\epsilon}{4}}}< 1$, for any finite-energy solution $u:S^{1}\times \mathbb{R}\rightarrow M$ of the Floer equation
$$
\partial_{s}u+ J(\partial_{t}u- X_{H})=0,
$$
the integral $\int u^{*}\tilde{\omega}\geq 0$. Here $J$ is a cylindrical almost complex structure compatible with $D$.
\end{lemma}
\begin{proof}
We prove this lemma by computation. By the definitions of $\tilde{\omega}$ and a cylindrical almost complex structure $J$, we know that $\tilde{g}(\cdot, \cdot):= \tilde{\omega}(\cdot, J\cdot)$ defines a Riemannian metric on the region where $\tilde{\omega}\neq 0$. Then by the above computation we have that
$$
\int u^{*}\tilde{\omega}= \int\tilde{\omega}(\partial_{s}u, \partial_{t}u)\geq \int \abs{\partial_{s}u}_{\tilde{g}}^{2}- \norm{H}_{M-D_{1+\frac{3\epsilon}{4}}}\geq -\norm{H}_{M-D_{1+\frac{3\epsilon}{4}}}.
$$
Since $\tilde{\omega}$ is supported in $M-D_{1+\frac{3\epsilon}{4}}$, we only need to consider the relative Hofer norm $\norm{H}_{M-D_{1+\frac{3\epsilon}{4}}}$. In particular, if the image of $u$ does not intersect $M-D_{1+\frac{3\epsilon}{4}}$, then $\int u^{*}\tilde{\omega}=0$.

Since the non-constant orbits of $H$ are in $D_{1+\frac{3\epsilon}{4}}$, the cylinder $u$ represents a relative homology class in $H_{2}(M, D)$. By our assumption that the number $\int u^{*}\tilde{\omega}$ is an integer, when $\norm{H}_{M-D_{1+\frac{3\epsilon}{4}}}< 1$, the number $\int u^{*}\tilde{\omega}$ is non-negative.
\end{proof}

Note that Lemma \ref{pos} also works for a family of Hamiltonian $\lbrace H_{t,s}\rbrace$ under small perturbations. That is, if $\norm{H_{t,s}-H_{t,s'}}_{M-D_{1+\frac{3\epsilon}{4}}}< 1$ for any $s, s'$ then we still have positivity of solutions of the parameterized Floer equation.

\subsection{A filtration which is not exhaustive}
The computation in the last subsection tells that when the Hamiltonian functions have small relative Hofer norms outside the Liouville domain, the outside part of the corresponding Floer solutions carry non-negative $\tilde{\omega}$-energy. Now we use this fact to define a filtration on the telescope, which is the underlying complex of the relative symplectic cohomology.

First we consider the case of a single $S$-shape Hamiltonian $H$ such that $\norm{H}<1$ on $M-D_{1+\frac{3}{4}\epsilon}$. We define a valuation of a single element
$$
a\cdot \gamma\in CF^{k}(H)= \bigoplus_{\gamma\in \mathcal{CP}^{k}(H)} \Lambda_{0}\cdot \gamma
$$
by setting
\begin{equation}
\sigma(a\cdot \gamma)=v(a)- \int_{\gamma}H -\int_{\gamma}\tilde{\theta}
\end{equation}
where $v(a)$ is the valuation on $\Lambda_{0}$. If $\gamma$ is a constant orbit, the integral $\int_{\gamma}\tilde{\theta}$ is zero. For a general sum $x=\sum_{i}a_{i} \gamma_{i}$ we define the valuation as
\begin{equation}
\sigma(x):=\inf_{i}\lbrace \sigma(a_{i} \gamma_{i})\rbrace =\inf_{i}\lbrace v(a_{i})- \int_{\gamma_{i}} H_{i}- \int_{\gamma_{i}}\tilde{\theta}\rbrace.
\end{equation}

\begin{lemma}
For the valuation $\sigma$ we have that
$$
\sigma(d(a\cdot \gamma))\geq \sigma(a\cdot \gamma).
$$
Hence it induces a filtration
\begin{equation}
\mathcal{F}^{\lambda} CF^{k}(H)= \lbrace x\in CF^{k}(H)\mid \sigma(x)\geq \lambda\rbrace
\end{equation}
on the complex $CF^{k}(H)$.
\end{lemma}
\begin{proof}
We prove this lemma by a direct computation. First from our definition
$$
d(a\cdot \gamma)=a\cdot \sum_{\gamma', A} n(\gamma, \gamma'; A)\cdot \gamma'\cdot T^{\omega(A)-\int_{\gamma}H+ \int_{\gamma'}H}
$$
where $A$ is the homotopy class represented by a solution $u$ connecting $\gamma$ and $\gamma'$. Hence we have that
\begin{equation}
\begin{aligned}
\sigma(d(a\cdot \gamma))=& \sigma(a\cdot \sum_{\gamma',A} n(\gamma, \gamma'; A)\cdot \gamma'\cdot T^{\omega(A)-\int_{\gamma}H \int_{\gamma'}H})\\
=& \inf_{\gamma',A}\lbrace v(a)+ \omega(A)-\int_{\gamma}H+ \int_{\gamma'}H- \int_{\gamma'}H- \int_{\gamma'}\tilde{\theta} \rbrace\\
=& \inf_{\gamma',A}\lbrace v(a)+ \omega(A)-\int_{\gamma}H- \int_{\gamma'}\tilde{\theta} \rbrace\\
=& \inf_{\gamma',A}\lbrace v(a)+ \tilde{\omega}(A)+ d\tilde{\theta}(A)- \int_{\gamma}H- \int_{\gamma'}\tilde{\theta}\rbrace\\
=& \inf_{\gamma',A}\lbrace v(a)+ \tilde{\omega}(A)+ \int_{\gamma'}\tilde{\theta}- \int_{\gamma}\tilde{\theta}- \int_{\gamma}H- \int_{\gamma'}\tilde{\theta}\rbrace\\
=& \inf_{\gamma',A}\lbrace v(a)+ \tilde{\omega}(A)- \int_{\gamma}\tilde{\theta}- \int_{\gamma}H\rbrace\\
=& \inf_{\gamma',A}\lbrace \sigma(a\cdot \gamma)+ \tilde{\omega}(A) \rbrace\\
\geq & \sigma(a\cdot \gamma).
\end{aligned}
\end{equation}
Here we use that $\omega=\tilde{\omega}+ d\tilde{\theta}$ and the Stokes formula to compute $d\tilde{\theta}(A)$. The last inequality uses Lemma \ref{pos}.

Similarly we can check that $\sigma(dx)\geq \sigma(x)$ for a sum $x=\sum_{i}a_{i} \gamma_{i}$. Hence this valuation $\sigma$ induces a decreasing filtration on $CF(H)$.
\end{proof}

The above computation also works for a one-parameter family of Hamiltonian functions, if the variation of these functions is sufficiently small outside the Liouville domain. More precisely, we pick a monotone one-parameter family of Hamiltonian functions $H_{n,t}$ which form an acceleration data to compute the relative symplectic cohomology of $D$, such that
	\begin{enumerate}
		\item The relative Hofer norm of $H_{n,t}$ on $M-D_{1+\frac{3}{4}\epsilon}$ is less than one for each $n$;
		\item The relative Hofer norm of $H_{n,t}- H_{n+1, t}$ on $M-D_{1+\frac{3}{4}\epsilon}$ is less than one for each $n$.
	\end{enumerate}
Then each of the continuation maps also satisfies the above lemma. Therefore the telescope given by
$$
\mathcal{CF}:= CF(H_{1, t})\rightarrow CF(H_{2, t})\rightarrow \cdots \rightarrow CF(H_{n, t})\rightarrow \cdots
$$
satisfies that its differentials and continuation maps are compatible with $\sigma$. We write $\widehat{tel}(\mathcal{CF})$ as the completion of the telescope of this one-ray. For a general element $x=\sum_{i}a_{i} \gamma_{i}\in \widehat{tel}(\mathcal{CF})$ we define $\sigma: \widehat{tel}(\mathcal{CF})\rightarrow \mathbb{R}\cup \lbrace -\infty\rbrace$ by
\begin{equation}\label{fil}
\sigma(x):=\inf_{i}\lbrace \sigma(a_{i} \gamma_{i})\rbrace =\inf_{i}\lbrace v(a_{i})- \int_{\gamma_{i}} H_{i,t}- \int_{\gamma_{i}}\tilde{\theta}\rbrace
\end{equation}
This valuation gives a filtration on $\widehat{tel}(\mathcal{CF})$. The differentials of the completed telescope are sums of the differentials and continuation maps in the Floer complex. Hence the differential of $\widehat{tel}(\mathcal{CF})$ is compatible with this filtration, which makes $\widehat{tel}(\mathcal{CF})$ a filtered differential graded module. The homology of $\widehat{tel}(\mathcal{CF})$ is the relative symplectic cohomology.

Note that for a fixed Hamiltonian $H$ this valuation takes values strictly in $\mathbb{R}$ since $H$ is bounded and there are finitely many periodic orbits. But for a general element in the completion of the telescope the valuation can be negative infinity. Therefore the induced filtration is \textit{not exhaustive}. That is,
$$
\bigcup_{\lambda} \mathcal{F}^{\lambda} \widehat{tel}(\mathcal{CF})\neq \widehat{tel}(\mathcal{CF})
$$
which is one main reason that the induced spectral sequence sometimes does not converge. We also remark that this filtration is \textit{weakly convergent} since
$$
\bigcap_{\lambda} \mathcal{F}^{\lambda} \widehat{tel}(\mathcal{CF})=\lbrace 0\rbrace.
$$

Now we recall two foundational theorems on spectral sequences from \cite{Mc}, one on the existence and the other on the convergence.

\begin{definition}
Let $R$ be a commutative ring with unit. An $R$-module $A$ is called a filtered differential graded module if
\begin{enumerate}
\item $A$ is a direct sum of submodules, $A=\oplus_{n=0}^{\infty}A^{n}$.
\item There is an $R$-linear map $d: A\rightarrow A$, satisfying $d\circ d=0$.
\item $A$ has a filtration $F$ and the differential $d$ respects the filtration, that is, $d: F^{p}A\rightarrow F^{p}A$.
\end{enumerate}
\end{definition}

\begin{theorem}\label{spec1}(Theorem 2.6, \cite{Mc})
Each filtered differential graded module $(A, d, F)$ determines a spectral sequence, $\lbrace E_{r}^{*, *}, d_{r}\rbrace, r=1, 2, \cdots$ with $d_{r}$ of bidegree $(r, 1-r)$ and
$$
E_{1}^{p,q}\cong H^{p+q}(F^{p}A/F^{p+1}A).
$$
\end{theorem}

\begin{theorem}\label{spec2}(Theorem 3.2, \cite{Mc})
Let $(A, d, F)$ be a filtered differential graded module such that the filtration is exhaustive and weakly convergent. Then the associated spectral sequence with $E_{1}^{p,q}\cong H^{p+q}(F^{p}A/F^{p+1}A)$ converges to $H(A, d)$, that is,
$$
E_{\infty}^{p,q}\cong F^{p}H^{p+q}(A, d)/F^{p+1}H^{p+q}(A, d).
$$
\end{theorem}

\section{The induced spectral sequence}\label{sec:ss}
Previously we constructed a filtration on $\widehat{tel}(\mathcal{CF})$, by using a special family of Hamiltonian functions that have small variations outside $D$. However the induced spectral sequence does not always converge since the filtration is not exhaustive. Now we study the particular case that $D$ is a Liouville domain with an index bounded boundary in a symplectic Calabi-Yau manifold $M$. 

First, we give an outline of the proof in this section. For a fixed degree $k$, we will study $SH^{k}_{M}(D)$. In practice we will study the Hamiltonian Floer cohomology of a lower semi-continuous function, such that it is a fixed non-degenerate Hamiltonian function on $D\times S^{1}$ and it is positive infinity outside. It is isomorphic to $SH^{k}_{M}(D)$ over the Novikov field. Abusing the notation, we still write it as $SH^{k}_{M}(D)$. Let $\widehat{tel}(\mathcal{CF})$ be the completed telescope, constructed by a special family of Hamiltonian functions as in the previous section. Pick a sequence
$$
0< E_{1}< E_{2}< \cdots< E_{l}< \cdots
$$
of positive numbers going to infinity. We have a sequence of telescopes $tel(\mathcal{CF})\otimes_{\Lambda_{0}}\Lambda_{E_{l}}$, which form an inverse system by using projection maps. The homology 
$$
SH^{k}_{M}(D; \Lambda_{E_{l}}):= H^{k}(tel(\mathcal{CF})\otimes_{\Lambda_{0}}\Lambda_{E_{l}})
$$ 
is called the truncated symplectic cohomology. Since projections are chain maps, we have an inverse system in the homology level
$$
\cdots \leftarrow SH^{k}_{M}(D; \Lambda_{E_{l-1}})\leftarrow SH^{k}_{M}(D; \Lambda_{E_{l}})\leftarrow \cdots.
$$
The inverse limit
$$
\varprojlim_{l} SH^{k}_{M}(D; \Lambda_{E_{l}})
$$
is called the \textit{reduced symplectic cohomology} in \cite{GroV}. The relation between this reduced symplectic cohomology and the relative symplectic cohomology can be studied by verifying certain Mittag-Leffler condition of the above inverse systems.

The goal of this section is the following.
\begin{enumerate}
	\item For each $E_{l}$, we will construct two chain models which both compute $SH^{k}_{M}(D; \Lambda_{E_{l}})$. They are defined by using particular Hamiltonians, which depend on $E_{l}$. The first chain model is a telescope, on which the filtration defined in the previous section is exhaustive. Hence we have a convergent spectral sequence for this chain model to compute $SH^{k}_{M}(D; \Lambda_{E_{l}})$. This will establish $(1), (2)$ in Theorem \ref{main}.
	\item The second chain model is a direct limit. It helps us to show $SH^{k}_{M}(D; \Lambda_{E_{l}})$ is finitely-generated, and the number of generators is independent of $E_{l}$, thanks to the index bounded condition. By using this, we can verify a finite homological torsion criterion in \cite{GroV}, which shows that the reduced symplectic cohomology is isomorphic to the relative symplectic cohomology. This proves $(3)$ in Theorem \ref{main}.
\end{enumerate}

\subsection{Ignoring upper orbits and convergence}

For a monotone family of $S$-shape Hamiltonian functions, the orbits form two groups: upper orbits and lower orbits. And the Floer differentials/continuation maps have four components: upper-to-upper, upper-to-lower, lower-to-upper and lower-to-lower. The next lemma says that, under the index bounded condition, any lower-to-upper Floer trajectory has a big topological energy.

\begin{lemma}\label{ge}
Given any energy bound $E>0$ and an integer $k$, there exists an $S$-shape Hamiltonian function such that any lower-to-upper Floer differential with degree-$k$ input has topological energy greater than $E$.
\end{lemma}
\begin{proof}
First, we choose an $S$-shape Hamiltonian function $H_{t}$ with $\norm{H}_{M-D_{1+\frac{3\epsilon}{4}}}$ less than $1$. For a Floer differential with degree-$k$ input $\gamma$ and degree-$(k+1)$ output $\gamma'$, the energy weight is
$$
\begin{aligned}
&\int u^{*}\omega -\int_{\gamma} H_{t} +\int_{\gamma'} H_{t}\\
=& -\int_{\gamma}\tilde{\theta} +\int_{\gamma'}\tilde{\theta} +\int u^{*}\tilde{\omega}-\int_{\gamma} H_{t} +\int_{\gamma'} H_{t}.
\end{aligned}
$$
By the index bounded condition, the difference $-\int_{\gamma}\tilde{\theta} +\int_{\gamma'}\tilde{\theta}$ is a bounded number which only depends on $k$. The computation in Lemma \ref{pos} says that $\int u^{*}\tilde{\omega}$ is non-negative. Therefore if we choose an $S$-shape Hamiltonian function such that its upper level is high enough, then the above energy weight is larger than $E$.
\end{proof}

Note that this estimate is uniform for all degree $k$ orbits, hence we can make a subcomplex which only contains lower orbits, which gives the following lemma.

\begin{lemma}
For any integer $k$ and an energy bound $E>0$, let $H$ be a Hamiltonian function which satisfies the Lemma \ref{ge} for all three integers $k-1, k, k+1$ for $E$.
Let $CF^{k, L}(H)$ be the subspace of $CF^{k}(H)$ which only contains lower orbits. Let $d$ be the restriction of the Floer differential to $CF^{k, L}(H)$. Then
$$
0\rightarrow CF^{k-1, L}(H) \xrightarrow{d} CF^{k, L}(H) \xrightarrow{d} CF^{k+1, L}(H)\rightarrow 0
$$
satisfies that $d\circ d=0$ over $\Lambda_{E}$.
\end{lemma}
\begin{proof}
Take $\gamma\in CF^{k-1, L}(H)$, the usual argument to show $d\circ d (\gamma)=0$ is to look at broken Floer trajectories. In our case, if it breaks along an upper orbit, then the energy weight is greater than $E$ by Lemma \ref{ge}, which is automatically zero over $\Lambda_{E}$. Hence we can ignore the upper orbits contributions.
\end{proof}

The above two lemmas tell us for any fixed energy bound $E$ and degree-$k$, we can use only lower orbits to form a homology theory. By the same argument, lower-to-upper continuation maps have big topological energy for particular Hamiltonians.

\begin{figure}
	\begin{tikzpicture}[xscale=0.6, yscale=0.6]
		\draw [->] (0,0)--(12,0);
		\draw [snake it] (0,-1)--(4,-1);
		\draw (4,-1) to [out=0, in=225] (5,0);
		\draw (5,0)--(7,3);
		\draw (7,3) to [out=45, in=180] (8,4);
		\draw [snake it] (8,4)--(10,4);
		\node at (11, 4) {$H_{n,t}$};

		\draw (5,0)--(7,5);
		\draw (7,5) to [out=45, in=180] (8,6);
		\draw [snake it] (8,6)--(10,6);
		\node at (11.5, 6) {$H_{n+1,t}$};
		
		\draw [dotted] (10,4)--(10,6);
		\node [right] at (10,5) {$\frac{1}{2}$};
		
	\end{tikzpicture}
	\caption{Hamiltonian functions with fixed lower parts and small variations outside}\label{Ham4}
\end{figure}

\begin{lemma}\label{conti}
Given any energy bound $E>0$ and an integer $k$, there exists a family of non-degenerate $S$-shape Hamiltonian functions $\lbrace
H_{n}\rbrace_{n\in \mathbb{N}}$ such that
\begin{enumerate}
\item $H_{1}\leq H_{2}\leq \cdots H_{n}\leq H_{n+1}\leq \cdots$;
\item $H_{1}= H_{2}= \cdots H_{n}= \cdots$ on $S^{1}\times D$;
\item $H_{1}$ satisfies the above two lemmas;
\item any lower degree-$k$ orbits are inside $D$;
\item any continuation map from a lower degree-$k$ orbit to an upper degree-$k$ orbit has topological energy greater than $E$.
\end{enumerate}
\end{lemma}
\begin{proof}
Take $H_{1}$ as a function which satisfies the above two lemmas. Then we construct $H_{2}$ in the following way
\begin{enumerate}
	\item $H_{1}=H_{2}$ on $D$;
	\item $H_{1}+\frac{1}{2}=H_{2}$ on $M-D_{1+\frac{3\epsilon}{4}}$;
	\item Near $\partial D$, the function $H_{2}$ is obtained from a cylindrical function by adding time-dependent perturbations. We assume the perturbation is small such that $H_{1}\leq H_{2}$ globally, see Figure \ref{Ham4}.
\end{enumerate} 

Then we repeat this process inductively to get $H_{n+1}$ from $H_{n}$. Hence the items $(1)-(4)$ are satisfied, and we can connect $H_{n}$ with $H_{n+1}$ by a monotone homotopy. If we have a continuation map from a lower degree-$k$ orbit $\gamma$ of $H_{n}$ to an upper degree-$k$ orbit $\gamma'$ of $H_{n+1}$, then it is weighted by an energy
$$
\begin{aligned}
	&\int u^{*}\omega -\int_{\gamma} H_{n} +\int_{\gamma'} H_{n+1} +\int (\partial_{s}H_{s})\\
	=& \int u^{*}\tilde{\omega} -\int_{\gamma} \tilde{\theta} +\int_{\gamma'} \tilde{\theta} -\int_{\gamma} H_{n} +\int_{\gamma'} H_{n+1} +\int (\partial_{s}H_{s}).
\end{aligned}
$$
By the construction, we have that $\norm{H_{n}- H_{n+1}}_{M-D_{1+\frac{3\epsilon}{4}}}<1$ hence the first term is non-negative. The monotonicity of $\partial_{s}H_{s}$ shows the last term is non-negative and the index bounded condition shows the second and third terms are bounded. So the whole energy is larger than $E$. Because the difference between the lower levels of $H_{n}$ and the upper level of $H_{n+1}$ are big enough. 
\end{proof}

Therefore by using the above family of Hamiltonian functions, we are computing the Hamiltonian Floer cohomology of a semi lower-continuous function, which is a non-degenerate Hamiltonian function on $D$ and is positive infinity outside $D$. By Proposition \ref{single}, the resulting invariant is isomorphic to the relative symplectic cohomology over the Novikov field.

Now we consider two direct systems
$$
\mathcal{CF}^{k}=CF^{k}(H_{1,t})\rightarrow CF^{k}(H_{2,t})\rightarrow \cdots
$$
and
$$
\mathcal{CF}^{k, L}=CF^{k, L}(H_{1,t})\rightarrow CF^{k, L}(H_{2,t})\rightarrow \cdots
$$
over $\Lambda_{E}$, induced by Hamiltonian functions defined in the above lemma. For each $n$, we have an inclusion of a sub-complex $CF^{k,L}(H_{n,t})\to CF^{k}(H_{n,t})$ over $\Lambda_{E}$. We recall the following algebraic property of a telescope of sub-complexes. Suppose that we have a commutative diagram of chain complexes
$$
\begin{tikzcd}
	\mathcal{C}':=
	& C_{1}'  \arrow[r] \arrow[d]
	& C_{2}'  \arrow[r] \arrow[d]
	& C_{3}'  \arrow[r] \arrow[d]
	& \cdots \\
	\mathcal{C}:=
	& C_{1} \arrow[r]
	&  C_{2} \arrow[r]
	&  C_{3} \arrow[r]
	& \cdots
\end{tikzcd}
$$
where horizontal maps are chain maps (which we call continuation maps) and vertical maps are inclusion maps of sub-complexes. Then we have an induced map between telescopes $tel(\mathcal{C}')\to tel(\mathcal{C})$.

\begin{lemma}[Lemma A.1 \cite{BSV}]\label{algebraic}
	Suppose that every $\gamma\in C_{n}$ there exists $N(\gamma)>0$ such that under continuation maps $\gamma$ lands in $C_{n+ N(\gamma)}'$, then $tel(\mathcal{C}')\to tel(\mathcal{C})$ is a quasi-isomorphism.
\end{lemma}

This lemma can be applied to our lower Floer complexes.

\begin{proposition}\label{ig}
	For any integer $k$ and energy bound $E>0$, there exists a family of Hamiltonian functions such that
	$$
	tel(\mathcal{CF}^{k, L})\to tel(\mathcal{CF}^{k})
	$$
	is a quasi-isomorphism over $\Lambda_{E}$.
\end{proposition}
\begin{proof}
	For any integer $k$ and energy bound $E>0$, we pick Hamiltonian functions as above to get a commutative diagram between $CF^{k,L}(H_{n,t})$ and $CF^{k}(H_{n,t})$ such that horizontal maps are continuations and vertical maps are inclusions. Next we check the condition in Lemma \ref{algebraic}.
	
	Pick $\gamma \in CF^{k}(H_{n,t})$, if it is a lower orbit, then its image under continuation maps are always lower, by $(5)$ in Lemma \ref{conti}. So we assume $\gamma$ is an upper orbit, and we will show after several continuation maps, it becomes lower or zero. Let $T^{E_{1}}\gamma_{1}$ be the image of $\gamma$ under the continuation map $CF^{k}(H_{n,t})\to CF^{k}(H_{n+1,t})$. Then
	$$
	E_{1}= -\int_{\gamma}\tilde{\theta} +\int_{\gamma_{1}}\tilde{\theta} +\int u_{1}^{*}\tilde{\omega} 	-\int_{\gamma}H_{n,t} +\int_{\gamma_{1}}H_{n+1,t}.
	$$
	If $\gamma_{1}$ is a lower orbit, then we are done. Otherwise we consider $T^{E_{2}}\gamma_{2}$ as the image of $T^{E_{1}}\gamma_{1}$ under the continuation map $CF^{k}(H_{n+1,t})\to CF^{k}(H_{n+2,t})$. We have 
	$$
	\begin{aligned}
		E_{2}=& E_{1}+ (-\int_{\gamma_{1}}\tilde{\theta} +\int_{\gamma_{2}}\tilde{\theta} +\int u_{2}^{*}\tilde{\omega} 	-\int_{\gamma_{1}}H_{n+1,t} +\int_{\gamma_{2}}H_{n+2,t})\\
		=& -\int_{\gamma}\tilde{\theta} +\int_{\gamma_{2}}\tilde{\theta} +\int u_{1}^{*}\tilde{\omega} +\int u_{2}^{*}\tilde{\omega} 	-\int_{\gamma}H_{n,t} +\int_{\gamma_{2}}H_{n+2,t}.
	\end{aligned}
	$$
	By the index bounded condition, the first two terms are bounded. Moreover, our Hamiltonian functions have small variations outside, the $\tilde{\omega}$-energy terms are non-negative. If $\gamma_{2}$ is still an upper orbit, we will consider its image under the third continuation map. Therefore after $N$ compositions of continuation maps, either $\gamma$ is sent to a lower orbit, or $E_{N}>E$, since $-\int_{\gamma}H_{n,t} +\int_{\gamma_{N}}H_{n+N,t}$ becomes arbitrarily large. This completes the proof by applying Lemma \ref{algebraic}.

\end{proof}

We call the process to get this quasi-isomorphism \textit{ignoring upper orbits}. And we define the truncated symplectic cohomology as
$$
SH^{k}_{M}(D; \Lambda_{E}):= H(tel(\mathcal{CF}^{k, L}); \Lambda_{E})= H(tel(\mathcal{CF}^{k}); \Lambda_{E}).
$$

Therefore Proposition \ref{ig} gives two ways to compute the truncated symplectic cohomology: by using $tel(\mathcal{CF}^{k, L})$ or $tel(\mathcal{CF}^{k})$. The former has the advantage that all generators of the underlying complex are local and the differentials are global. Moreover, the filtration on $tel(\mathcal{CF}^{k})$ constructed in Section \ref{sec:fil} becomes exhaustive when restricted to $tel(\mathcal{CF}^{k, L})$. Hence we will get a convergent spectral sequence to compute the truncated symplectic cohomology. This will be discussed in next subsection.  Now we study how this truncated symplectic cohomology is related to the usual relative symplectic cohomology.

First, note that a telescope is quasi-isomorphic to a direct limit. Hence we also have that
$$
SH^{k}_{M}(D; \Lambda_{E})= H(tel(\mathcal{CF}^{k, L}); \Lambda_{E})= H(\varinjlim_{n}\mathcal{CF}^{k, L}; \Lambda_{E}).
$$
The following lemma shows that the direct limit has a simpler description in our index bounded case.

\begin{lemma}\label{finite}
For a fixed degree $k$, the direct limit $\varinjlim_{n}\mathcal{CF}^{k, L}$ which contains lower orbits is a finite-dimensional free $\Lambda_{E}$-module.
\end{lemma}
\begin{proof}
By the index bounded condition and non-degeneracy of the contact form, for each fixed Hamiltonian $CF^{k, L}(H_{n,t})$ is a finite-dimensional free $\Lambda_{E}$-module, generated by degree-$k$ orbits. On the other hand, the lower parts of the Hamiltonian functions $H_{n,t}$ are fixed. Hence we have a canonical identification between $CF^{k, L}(H_{n,t})$ for different $n$. Next we study the continuation maps between $CF^{k, L}(H_{n,t})$. After identifying the generators of $CF^{k, L}(H_{n,t})$ for different $n$, the continuation maps can be written as $l\times l$ matrices $\lbrace a^{n}_{ij}\rbrace$ with entries $a^{n}_{ij}\in \Lambda_{E}$, where $l$ is the dimension of $CF^{k, L}(H_{n,t})$.

Let $u$ be a Floer cylinder contributing to the continuation maps, and assume $u$ is contained in the region where the Hamiltonian function is fixed. Then by regularity it can only be the identity map, see Lemma \ref{translation}. There may be other Floer cylinders that travel outside $D$ and contribute to the continuation maps, which have non-trivial topological energy. Hence the continuation map, viewed as a matrix, have the following properties:
\begin{enumerate}
\item the entries on the diagonal are $a^{n}_{jj}= 1+ b^{n}_{jj}$ with $v(b^{n}_{jj})>0$;
\item the off-diagonal entries have strictly positive valuations.
\end{enumerate}
Note that the determinant of the matrix $\lbrace a^{n}_{ij}\rbrace$ has a constant term one, which says that the matrix is invertible. Hence each continuation map is an isomorphism of $\Lambda_{0}$-modules. So the direct limit $\varinjlim_{n}\mathcal{CF}^{k, L}$ is isomorphic to $CF^{k, L}(H_{n,t})$, which is a finite-dimensional free $\Lambda_{E}$-module.
\end{proof}

To effectively use the truncated symplectic cohomology, there are two options: to show it is an invariant with good properties, or to relate it with the original relative symplectic cohomology $SH^{k}_{M}(D)$.

For the first option, we expect the following is true

\begin{proposition}
The truncated symplectic cohomology $SH^{k}_{M}(D; \Lambda_{E})$ is a finite-dimensional $\Lambda_{E}$-module. Let $\lambda >0$ be the smallest number such that
$$
T^{\lambda}\cdot SH^{k}_{M}(D; \Lambda_{E})=0.
$$
Then the displacement energy of $D$ is not less than $\lambda$.
\end{proposition}
\begin{proof}
The finite-dimensionality follows from the above lemma. We expect the proof of the energy relation is similar to Remark 4.2.8 in \cite{Var} in the original relative symplectic cohomology setting. The full proof will be pursued in the future.
\end{proof}

The number $\lambda$ is an analogue of the torsion threshold of the Lagrangian Floer cohomology (see Theorem J in \cite{FOOO}), which is related to the displacement energy of a Lagrangian submanifold. Since the energy bound $E$ can be as large as one needs, this proposition is useful for most displacement problems.

For the second option, we have the following

\begin{proposition}\label{recover}
The inverse limit of the truncated symplectic cohomology recovers the original relative symplectic cohomology. That is,
$$
(\varprojlim_{E}SH^{k}_{M}(D; \Lambda_{E}))\otimes_{\Lambda_{0}}\Lambda \cong SH^{k}_{M}(D)\otimes_{\Lambda_{0}}\Lambda.
$$
Here the inverse limit is taken as $E$ goes to infinity.
\end{proposition}

To prove the proposition, first we recall some results in homological algebra.

\begin{definition}
An inverse system
$$
\mathcal{C}= C_{1}\leftarrow C_{2}\leftarrow \cdots
$$
is said to satisfy the \textit{Mittag-Leffler condition} if for each $n\in \mathbb{N}$, there exists $i\geq n$ such that for all $j\geq i$, we have that
$$
\text{Im}(C_{j}\rightarrow C_{n})\cong \text{Im}(C_{i}\rightarrow C_{n}).
$$
\end{definition}

The Mittag-Leffler condition shows the vanishing of the $\varprojlim^{1}$ of an inverse system.

\begin{proposition}(Proposition 3.5.7 \cite{W})
If an inverse system $\mathcal{C}$ satisfies the Mittag-Leffler condition, then $\varprojlim^{1}(\mathcal{C})=0$.
\end{proposition}

\begin{proposition}(Proposition 3.5.8 \cite{W})
For an inverse system
$$
\mathcal{C}= C_{1}\leftarrow C_{2}\leftarrow \cdots
$$
of complexes, which satisfies the degree-wise Mittag-Leffler condition, we have an short exact sequence
$$
0\rightarrow \varprojlim\nolimits^{1} H^{*}(C_{n}) \rightarrow H^{*}(\varprojlim C_{n}) \rightarrow \varprojlim H^{*}(C_{n}) \rightarrow 0.
$$
\end{proposition}

Now we begin the proof of Proposition \ref{recover}.

\begin{proof}
Fix a sequence of positive numbers
$$
0< E_{1}< E_{2}< \cdots <E_{l}< \cdots
$$
going to infinity. Let $\lbrace G_{n}\rbrace$ be a fixed acceleration data we used in Proposition \ref{ig} and Lemma \ref{finite}. Then we have that,
$$
SH^{k}_{M}(D)\otimes_{\Lambda_{0}} \Lambda= H(\varprojlim_{E_{l}} tel(CF^{k}(G_{n}) \otimes_{\Lambda_{0}} \Lambda_{E_{l}}))\otimes_{\Lambda_{0}} \Lambda.
$$
The complexes $tel(CF^{k}(G_{n}) \otimes_{\Lambda_{0}} \Lambda_{E_{l}})$ form an inverse system over $E_{l}$. Note that the Hamiltonian functions are independent of $E_{l}$, so the maps in this system are given by projection. Since projections are surjective, this inverse system satisfies the degree-wise Mittag-Leffler condition. Hence we have a short exact sequence
$$ 
0\rightarrow \varprojlim\nolimits^{1} H(C_{l}) \rightarrow H(\varprojlim C_{l}) \rightarrow \varprojlim H(C_{l}) \rightarrow 0,
$$
where $C_{l}:= tel(CF^{k}(G_{n}) \otimes_{\Lambda_{0}} \Lambda_{E_{l}})$. In this short sequence, the middle term is what we need to compute and the right term is the inverse limit of truncated symplectic cohomology. Hence it suffices to show the vanishing of the left term.

In Lemma \ref{finite}, we have showed that $H(\varinjlim_{n}\mathcal{CF}^{k, L}; \Lambda_{E_{l}})$ is finite-dimensional by using acceleration datum depending on $E_{l}$. On the other hand, the quasi-isomorphisms before Lemma \ref{finite} shows that
$$
H(C_{l})= H(\varinjlim_{n}\mathcal{CF}^{k, L}; \Lambda_{E_{l}}).
$$

So $H(C_{l})$ is finite-dimensional for any $E_{l}$. Moreover, for different $E_{l}$ the defining Hamiltonians for $\mathcal{CF}^{k, L}$ have the same fixed lower part. Hence the dimensions of $H(C_{l})$ have a uniform upper bound, independent of $l$, given by the number of lower degree $k$ orbits of Hamiltonian functions with a fixed lower part. Therefore $H(C_{l})$ over $l$ is an inverse system of finite-dimensional modules with a uniform upper bound on ranks. In the following we will show it satisfies the degree-wise Mittag-Leffler condition, which completes the proof. 
\end{proof}

Now we prove that $H(C_{l})$ over $l$ is an inverse system which satisfies the degree-wise Mittag-Leffler condition, by using Groman-Varolgunes' finite torsion criterion in \cite{GroV}.

Let $V$ be a $\Lambda_{0}$-module. For any element $v\in V$ we define 
$$
\tau(v):= \inf\lbrace \lambda\geq 0\mid T^{\lambda}v=0\rbrace
$$
and we define the \textit{maximal torsion} of $V$ as
$$
\tau(V):= \sup_{v\in V, \tau(v)< +\infty} \tau(v),
$$
see Definition 6.15 in \cite{GroV}. The following is a combination of Lemma 6.19 and Proposition 6.12 in \cite{GroV}. The invariant $SH^{k}_{M, \lambda}(D)$ in \cite{GroV} is our truncated invariant $SH^{k}_{M}(D; \Lambda_{E}))$ with $E=\lambda$.

\begin{proposition}
	If $SH^{k}_{M}(D)$ has finite maximal torsion, then  $H(C_{l})$ over $l$ is an inverse system which satisfies the degree-wise Mittag-Leffler condition.
\end{proposition}

Next we verify that $SH^{k}_{M}(D)$ has a finite maximal torsion. Suppose $SH^{k}_{M}(D)$ has an infinite maximal torsion. Then we have a sequence of elements $x_{n}, y_{n}$ in the completed telescope such that
$$
d(x_{n})=0, \quad d(y_{n})=T^{\tau_{n}}x, \quad 0<\tau_{1}<\tau_{2} \cdots<\tau_{n}< \cdots \to +\infty.
$$
This shows that as $l$ going to infinity, the number of $x_{n}$'s with valuations less than $E_{l}$ goes to infinity. Hence the rank of $H(C_{l})$ also goes to infinity, which contradicts to that their ranks are uniformly bounded from above.

\subsection{The first page of the spectral sequence}
We showed that the truncated symplectic cohomology recovers the relative symplectic cohomology. Now we study how to compute the truncated symplectic cohomology.

For any integer $k$ and energy bound $E>0$, we have three chain models
$$
tel(\mathcal{CF}^{k}), \quad tel(\mathcal{CF}^{k,L}), \quad \varinjlim_{n}\mathcal{CF}^{k, L}
$$
given by a particular family of Hamiltonian functions. They are all quasi-isomorphic hence they all compute the truncated symplectic cohomology over $\Lambda_{E}$. Now we equip the second chain model with the filtration defined in (\ref{fil}). Recall that for a general element $x=\sum_{i}a_{i} \gamma_{i}\in tel(\mathcal{CF}^{k,L})$ we define $\sigma: tel(\mathcal{CF}^{k,L})\rightarrow \mathbb{R}\cup \lbrace -\infty\rbrace$ by
\begin{equation}\label{val}
	\sigma(x):=\inf_{i}\lbrace v(a_{i})- \int_{\gamma_{i}} H_{i,t}- \int_{\gamma_{i}}\tilde{\theta}\rbrace.
\end{equation}
And for any $p\in\mathbb{R}$ we define
\begin{equation}\label{fil2}
	F^{p}tel(\mathcal{CF}^{k,L}):= \lbrace x\in tel(\mathcal{CF}^{k,L})\mid \sigma(x)\geq p\rbrace.
\end{equation}
By the computations in Section \ref{sec:fil}, we know that the differentials in the telescope are compatible with this filtration, which makes $tel(\mathcal{CF}^{k,L})$ a filtered differential graded module. Moreover, since all generators in $tel(\mathcal{CF}^{k,L})$ are lower orbits, the Hamiltonian terms in $\sigma(x)$ are uniformly bounded. By the index bounded condition, the integrals of $\tilde{\theta}$ in $\sigma(x)$ are also uniformly bounded. Hence $\sigma(x)> -\infty$ for any $x$. This shows that the filtration is exhaustive. (Actually this filtration is bounded from below by some number.) Therefore Theorem \ref{spec1} and \ref{spec2} give us a spectral sequence which converges to the truncated symplectic cohomology. This proves the convergence part in Theorem \ref{main} $(1)$. In the following we will compute the first page of this spectral sequence.

First we observe that by using the special family of $S$-shape Hamiltonians, all Floer cylinders which are not contained in the Liouville domain has positive $\tilde{\omega}$-energy, given that the asymptotic boundaries of the cylinders are lower orbits.

This will be proved via Abouzaid-Seidel's integrated maximum principle. We follow Section 3.4 \cite{BSV} and recall the setup now. Let $(K, \omega)$ be a symplectic manifold with a concave boundary $(Y, \alpha)$. That is, there is a symplectic embedding $(Y\times [c, c+\epsilon), d(r\alpha))$ onto a neighborhood of $Y$ in $K$, where $r$ is the Liouville coordinate. We will consider maps $u: (\Sigma, \partial \Sigma)\to (K, Y)$ solving the Floer equation with respect to a certain class of almost complex structures and Hamiltonian perturbations. Here $(\Sigma, \partial \Sigma)$ is a general Riemann surface with boundary.

To define our Floer equation, we choose a family of compatible almost complex structures $J_{z}$ parameterized by $z\in \Sigma$, and a Hamiltonian-valued one-form $\kappa\in \Omega^{1}(\Sigma; C^{\infty}(K))$. Note that we may interpret $\kappa$ as a one-form on $\Sigma \times K$, the de Rham differential has a decomposition
$$
d:=d_{\Sigma \times K}= d_{\Sigma}+ d_{K}.
$$
Then the non-degeneracy of $\omega$ gives us a Hamiltonian-vector-field-valued one-form $X_{\kappa}\in \Omega^{1}(\Sigma; C^{\infty}(TW))$. The Floer equation in consideration is
$$
(du- X_{\kappa})^{0, 1}= 0.
$$

\begin{proposition}[Proposition 3.9 \cite{BSV}]\label{imp}
	Suppose that
	\begin{enumerate}
		\item $J_{z}$ is of contact type along $Y$ for all $z\in \partial \Sigma$, that is, $dr\circ J_{z}= -r\alpha$;
		\item there exist one-forms $\beta_{1}, \beta_{2}\in \Omega^{1}(\Sigma)$ such that $\kappa= \beta_{1}\cdot r+ \beta_{2}$ in a neighborhood of $Y$;
		\item $d_{\Sigma}\kappa -\lbrace \kappa, \kappa\rbrace - d\beta_{2} \geq 0$.\footnote{Here $\lbrace \kappa, \kappa\rbrace$ lives in $\Omega^{2}(\Sigma; C^{\infty}(K))$ and is defined by $\lbrace \kappa, \kappa\rbrace(v,w):= \lbrace \kappa(v), \kappa(w)\rbrace$ by the Poisson bracket.}
	\end{enumerate}
	Then any smooth map $u: (\Sigma, \partial \Sigma \neq \emptyset)\to (K, Y)$ solving $(du- X_{\kappa})^{0, 1}= 0$, will satisfy that 
	$$
	\int_{\Sigma} u^{*}\omega -\int_{\partial\Sigma} u^{*}(c\alpha) \geq 0,
	$$ 
	with equality if and only if $u$ is contained in $Y$.
\end{proposition}

Next we apply the above proposition to our setting. We make two more assumptions on the $S$-shape Hamiltonian functions which are used to define the telescope:
\begin{enumerate}
	\item Near the boundary of $D_{1+\frac{3}{4}\epsilon}$, the function $H_{1}$ is a function of the radial coordinate with small positive slope;
	\item All $H_{n}$'s and all homotopies connecting them are given by translations in the $s$-coordinate outside $D_{1+\frac{3}{4}\epsilon}$. That is, we choose a smooth non-decreasing function $\phi(s): \mathbb{R}\to \mathbb{R}$ such that 
	$$
	\phi(s)=0, \forall s\leq 0, \quad \phi(s)=\frac{1}{2}, \forall s\geq 1.
	$$
	Then define $H_{n}= H_{1}+ n/2$ and define the homotopy between $H_{n}$ and $H_{n+1}$ to be  $H_{s}= H_{n}+ \phi(s)$ outside $D_{1+\frac{3}{4}\epsilon}$.
\end{enumerate}

\begin{proposition}
Let $D$ be a Liouville domain in $M$. Assume that $[\tilde{\omega}]$ is integral, see the paragraph after Lemma \ref{aux}. And assume our acceleration data satisfies the above conditions. Let $\gamma_{-}, \gamma_{+}$ be two one-periodic orbits of $H_{-}, H_{+}$ in the telescope which are lower orbits. Let $u$ be a Floer solution connecting $\gamma_{-}, \gamma_{+}$. If $Im(u)\cap (M-D_{1+\frac{3\epsilon}{4}})\neq \emptyset$, then $\tilde{\omega}(u)> \hbar$ where $\hbar$ is a positive number independent of $H_{\pm}$.
\end{proposition}
\begin{proof}
Suppose that $Im(u)\cap (M-D_{1+\frac{3\epsilon}{4}})\neq \emptyset$, we pick a number $\epsilon'$ which is slightly greater than $\epsilon$ and $\partial D_{1+\frac{3}{4}\epsilon'}$ intersects $Im(u)$ transversally. The intersection is a disjoint union of circles in $\mathbb{R}\times S^{1}$, which is non-empty when $\epsilon'$ is close to $\epsilon$.

Then we set $M-D_{1+\frac{3\epsilon'}{4}}$ to be the concave manifold $K$ in Proposition \ref{imp}, and $c=1+\frac{3\epsilon'}{4}$. By the transversal intersection we get a new map $u': (\Sigma, \partial \Sigma \neq \emptyset)\to (K, Y)$ solving the Floer equation. We will check the hypotheses in Proposition \ref{imp}. First in our definition of the telescope we did use contact type almost complex structures in the cylindrical region. Hence $(1)$ is satisfied.

In our case the one-form $\kappa= H_{s}dt$. Near $Y$ it is given by linear Hamiltonian functions, in terms of the $r$-coordinate: $\kappa= (ar+ b+ \phi(s))dt$. Hence $(2)$ is satisfied, with $\beta_{1}=adt, \beta_{2}= (b+\phi(s))dt$.

Next we verify $(3)$. Since there is no $ds$ term in $\kappa$, we have that $\lbrace \kappa, \kappa\rbrace(\partial_{s}, \partial_{t})=0$. Moreover, we can compute that
$$
\begin{aligned}
	d_{\Sigma}\kappa- d\beta_{2}&= \partial_{s} H_{s}dsdt- d\beta_{2}\\
	&= \phi'(s)dsdt- \partial_{s} (b+\phi(s))dsdt \\
	&= \phi'(s)dsdt- \phi'(s)dsdt= 0.
\end{aligned}
$$
Therefore all hypotheses in Proposition \ref{imp} are satisfied and we get
$$
0< \int_{\Sigma} u'^{*}\omega -\int_{\partial\Sigma} u'^{*}(c\alpha)= \int_{\Sigma} u'^{*}\omega -\int_{\partial\Sigma} u'^{*}\tilde{\theta}.
$$
By our integrality assumption, the right hand side of the equation is an integer. On the other hand, since the support of $\tilde{\omega}$ is outside $D_{1+\frac{3}{4}\epsilon}$, the integral
$$
\int u^{*}\tilde{\omega}= \int u^{*}\omega -\int u^{*}d\tilde{\theta}
$$
can be approximated by 
$$
\int_{\Sigma} u'^{*}\omega -\int_{\partial\Sigma} u'^{*}\tilde{\theta}= \int_{\Sigma} u'^{*}\omega -\int_{\partial\Sigma} u'^{*}(c\alpha) > 0
$$ 
as $\epsilon'$ tends to $\epsilon$. This shows that If $Im(u)\cap (M-D_{1+\frac{3\epsilon}{4}})\neq \emptyset$, then $\tilde{\omega}(u)\geq 1$. Finally, for example, we can set $\hbar= \frac{1}{2}$.
\end{proof}

\begin{remark}
	The above proposition is an analogue of Proposition 5.10 in \cite{BSV}. However our situation is easier since we can assume the Hamiltonian functions are translations of a fixed one outside $D_{1+\frac{3}{4}\epsilon}$.
\end{remark}

Now we can use this positive number $\hbar$ to construct a $\mathbb{Z}$-valued filtration out of the $\mathbb{R}$-valued filtration (\ref{fil2}). And for any $l\in\mathbb{Z}$ we define
\begin{equation}\label{fil3}
	F^{l}tel(\mathcal{CF}^{k,L}):=  \lbrace x\in tel(\mathcal{CF}^{k,L})\mid \sigma(x)\geq l\hbar\rbrace.
\end{equation}
Then this new filtration induces a convergent spectral sequence. Its first page is calculated by using all differentials of which the change of the $\tilde{\omega}$-energy are less than $\hbar$. On the other hand, Floer solutions that are not contained in $D$ are weighted by a positive $\tilde{\omega}$-energy greater than $\hbar$. Hence the first page of this spectral sequence is calculated by all differentials and continuation maps that are in $D$. Since in $D$ our Hamiltonian function is a fixed function, the continuation maps are identity maps and the differentials give the classical symplectic cohomology $SH^{k}(D; \Lambda_{E})$. In particular, when $[\tilde{\omega}]=0\in H^{2}(M, D; \mathbb{Z})$, there will not be outside contribution hence the spectral sequence degenerates at its first page. Now we have completed the proof of $(1)$ and $(3)$ in Theorem \ref{main}.

\section{Examples and extensions}\label{sec:examples}
We now discuss some applications of our spectral sequence, and how to construct perturbations in the Morse-Bott case.  

\begin{example}Let $B$ be a round ball symplectically embedded in a Calabi-Yau manifold $M^{2n}$ with an integral symplectic form. Then the boundary $\partial B$ carries the standard contact structure on an odd-dimensional sphere, which is Morse-Bott index bounded. After perturbing, the non-degenerate Reeb orbits on the sphere all have positive Conley-Zehnder indices. Our Hamiltonian flow is the reverse of the Reeb flow in the cylindrical region. Hence the degrees of non-constant Hamiltonian orbits are all less than $2n$. Moreover, we can choose the fixed lower part of our Hamiltonians does not have degree $2n$ constant orbits. So the only degree $2n$ constant orbits are upper constant orbits. Then we apply the ignoring upper orbits process to get that $SH^{2n}_{M}(B) \otimes_{\Lambda_{0}}\Lambda=0$. Finally, the Mayer-Vietoris sequence shows that any neighborhood of $M-B$ has non-vanishing relative symplectic cohomology, hence it is not stably-displaceable. This fact is already known in Theorem 1.1 \cite{I} and Corollary 1.15 \cite{TVar} by using different methods (with stronger conclusions). We put our argument here to motivate Proposition \ref{comp}.
\end{example}

\subsection{Simply-connected Lagrangians in Calabi-Yau manifolds}
Let $(L, g)$ be a Riemannian manifold and let $T^{*}L$ be its cotangent bundle with the standard symplectic form. The unit disk bundle $D_{1}T^{*}L$, with respect to the metric $g$, is a Liouville domain with the unit sphere bundle $ST^{*}L$ being its contact boundary. A closed geodesic $q(t)$ in $L$ lifts to a Reeb orbit $\gamma(t)=(q(t), q'(t))$ in $ST^{*}L$. Pick a trivialization $\Phi$ of the contact distribution along $\gamma$, there is a Conley-Zehnder index $CZ_{\Phi}(\gamma)$. On the other hand, the trivialization $\Phi$ also gives a trivialization of the symplectic vector bundle $TT^{*}L$ along $q$. Hence there is a Maslov index $\mu_{\Phi}(q)$ of $q$, viewed as a loop in $L$. The relation between these two indices is the following lemma.

\begin{lemma}\label{correspondence}(Lemma 2.1, \cite{CM})
In the above notations, we have that
$$
CZ_{\Phi}(\gamma)+ \mu_{\Phi}(q)= ind(q)
$$
where $ind(q)$ is the Morse index of $q$ as a geodesic.
\end{lemma}

Now we consider a Lagrangian submanifold $L$ in a Calabi-Yau manifold $M$. Let $D$ be a Weinstein neighborhood of $L$, which is isomorphic to $D_{1}T^{*}L$ for some metric $g$ on $L$. Then $\partial D$ is a contact hypersurface in $M$. In particular, when $L$ is simply-connected we can use trivializations induced by disk cappings to compute these indices. The Maslov index is always zero and we have that $CZ_{\Phi}(\gamma)= ind(q)\geq 0$.

Similar to the index bounded condition, we can consider the following relation between the Morse index and the length of a closed geodesic.

\begin{definition}
	A Riemannian metric $g$ is called \textit{index bounded} if for every $m>0$, there exists $\mu_{m}>0$ so that
	$$
	\lbrace length(q)\mid -m< ind(q)< m\rbrace \subset (0, \mu_{m})
	$$
	for all closed geodesics. And the metric is called (Morse-Bott) non-degenerate if the length functional is (Morse-Bott) non-degenerate.
\end{definition}

Hence if $L$ admits an index bounded Riemannian metric $g$ then any Lagrangian embedding of $L$ into a Calabi-Yau manifold admits an index bounded neighborhood. This is true for Riemannian manifolds with a positive Ricci curvature.

\begin{lemma}(see Theorem 19.4 and 19.6 \cite{M})
	Let $(L, g)$ be a closed Riemannian manifold of dimension $n$ whose Ricci curvature satisfies $Ric_{g}\geq (n-1)C$ for some positive real number $C$. Then any closed geodesic on $L$ with length $\lambda$ has Morse index greater than $\lambda\sqrt{C}/\pi -1$. In particular, $g$ is index bounded.
\end{lemma}
\begin{proof}
	Let $\gamma$ be a closed geodesic on $L$ with length $\lambda$, for the constant $C$ there exists an integer $l$ such that
	$$
	l\pi/\sqrt{C}< \lambda \leq(l+1)\pi/\sqrt{C}.
	$$
	We cut $\gamma$ into $l+1$ segments such that each of the first $l$ segments has length slightly greater than $\pi/\sqrt{C}$. Then by the proof of Theorem 19.6 \cite{M}, any geodesic segment with length greater than $\pi/\sqrt{C}$ is unstable. Hence each of these $l$ segments has index at least one, the index of $\gamma$ is at least $l$.
	
	Now let $\gamma$ be a closed geodesic with index $k$ and length $\lambda$, we know that $k> \lambda\sqrt{C}/\pi -1$ which shows that $\lambda< (k+1)\pi/\sqrt{C}$. Hence $g$ is index bounded.
\end{proof}

\begin{remark}
	The positivity of the Ricci curvature of a metric $g$ is preserved under $C^{\infty}$-small perturbations. Hence we obtain a non-degenerate index bounded metric $g_{\epsilon}$ after perturbation.
\end{remark}

Let $(L, g_{\epsilon})$ be a closed Riemannian manifold with a positive Ricci curvature. Then $g_{\epsilon}$ is a non-degenerate index bounded metric. Suppose there is a Lagrangian embedding $L\rightarrow M$ into a Calabi-Yau manifold $M$. Another consequence of the Bonnet-Myers theorem (Theorem 19.6 \cite{M}) is that $L$ has a finite fundamental group. Hence $L$ is a Lagrangian submanifold of $M$ with a vanishing Maslov class. Now let $U$ be a Weinstein neighborhood of $L$ induced by the metric $g_{\epsilon}$, the above lemma tells us that $U$ is a Liouville domain with a non-degenerate index bounded boundary in $M$. In conclusion, our spectral sequence works for computing $SH_{M}(U)$, given the integrality condition on $\tilde{\omega}$. 

If we have two Riemannian manifolds both with positive Ricci curvature, then the product manifold also admits a metric with positive Ricci curvature. Hence we can also study Lagrangian submanifolds of product type.

Now we present the example of spheres, to demonstrate the above method. Let $S^{n}$ be a sphere with dimension $n\geq 3$, and let $g_{R}$ be the round metric on $S^{n}$. It is known that $g_{R}$ is a Morse-Bott non-degenerate index bounded metric. Next let $g_{\epsilon}$ be a $C^{\infty}$-small generic perturbation of $g_{R}$, such that it is a non-degenerate index bounded metric.

The above discussion tells us that for any Lagrangian sphere $S=S^{n}$ with $n\geq 3$ in a Calabi-Yau manifold $M$ with an integral symplectic form, a Weinstein neighborhood $U$ of $S$ induced by the metric $g_{\epsilon}$ has a non-degenerate index bounded contact boundary. Hence our spectral sequence works for computing $SH_{M}(U)$.

\begin{lemma}
For $(M, S, U)$ as above, we have that $SH_{M}(U)\otimes_{\Lambda_{0}}\Lambda \neq 0$.
\end{lemma}
\begin{proof}
For a given energy bound $E>0$, there is a convergent spectral sequence which starts from the symplectic cohomology $SH(T^{*}S^{n}; \Lambda_{E})$ and converges to $SH_{M}(U; \Lambda_{E})$. In our degree notation (\ref{grading}), the usual symplectic cohomology $SH(T^{*}S^{n}; \mathbb{C})$ is non-zero and one-dimensional in degrees
$$
\lbrace n\rbrace\cup \lbrace i(1-n)+n, i(1-n)+n-1\mid i\in \mathbb{Z}_{+}\rbrace.
$$
(Note that our Hamiltonian flow is the reverse of the Reeb flow in the cylindrical region.) The non-zero element in degree-$n$ cannot be killed in the spectral sequence when $n\geq 3$, since the differential only changes the degree by one.

Hence for any energy bound $E>0$, the truncated invariant satisfies that
$$
SH_{M}^{n}(U; \Lambda_{E})\cong \Lambda_{E}.
$$
Then by taking the inverse limit over $E$, we have that $SH_{M}(U)\otimes_{\Lambda_{0}}\Lambda \neq 0$.
\end{proof}

So any Lagrangian sphere with dimension $n\geq 3$ in a Calabi-Yau manifold is stably non-displaceable. This is known by using the Lagrangian Floer cohomology of $S$, see Theorem L \cite{FOOO}. But by using the Mayer-Vietoris property, we can get more from the above lemma. Note that $SH_{M}^{2}(U)$ is zero hence cannot be isomorphic to the quantum cohomology of $M$. Pick a Weinstein neighborhood $V$ of $S$, also induced by the same $g_{\epsilon}$ but with a smaller radius in the fiber direction, compared with $U$. Then $(M-V)\cup U=M$ and the boundaries of $U, V$ do not intersect. The Mayer-Vietoris property says that $SH_{M}^{2}(M-V)\otimes_{\Lambda_{0}}\Lambda\neq 0$. 

\begin{lemma}\label{com sphere}
	For $(M, S, U)$ as above, we have that $M-U$ is stably non-displaceable.
\end{lemma}
\begin{proof}
	Suppose that $M-U$ is stably displaceable. Then a neighborhood $K$ of $M-U$ is also stably displaceable. We can choose $V$ as above such that $M-V\subset K$, contradicting to $SH_{M}^{2}(M-V)\otimes_{\Lambda_{0}}\Lambda\neq 0$.
\end{proof}

This result is new and it contrasts to the case where the ambient space is not Calabi-Yau: a Weinstein neighborhood $U$ of a Lagrangian sphere can be compactified to be a quadric hypersurface $Q^{n}\subset \mathbb{C}P^{n+1}$. Note that $Q^{n}$ is a monotone symplectic manifold, and the complement of $U$ in $Q^{n}$ is a small neighborhood of a divisor $Q^{n-1}$, which is stably displaceable.

\begin{example}
Let $L=S^{2}\times S^{3}$ be a Lagrangian submanifold in a Calabi-Yau manifold $M$. Then there exists a Weinstein neighborhood of $L$ which is a Liouville domain with a non-degenerate index bounded boundary. So it's possible to use our spectral sequence to compute $SH_{M}(L)$, which could determine the displaceability of $L$. However, due to the $S^{2}$ factor, we don't have an immediate non-vanishing result, compared with the case of spheres with dimension larger than two. Hence the $SH_{M}(L)$ may depend on the ambient space.

On the other hand, the Lagrangian Floer cohomology of $L$ may have obstructions to be defined. The obstruction lies in $H^{2}(L; \mathbb{Q})$, see Theorem L \cite{FOOO}.
\end{example}

Another application of the geodesic-Reeb orbit correspondence is a generalization of Lemma \ref{com sphere}. If we only care about the complement of the Lagrangian, then no index bounded Riemannian metric is needed.

\begin{proposition}\label{comp}
	Let $(M^{2n}, \omega)$ be a symplectic Calabi-Yau manifold with $n>2$ and $\omega$ represents an integral class in $H^{2}(M)$. For a simply-connected Lagrangian $S$ in $M$ and a Weinstein neighborhood $U$ of $S$, we have that $M-U$ is not stably-displaceable in $M$.
\end{proposition}
\begin{proof}
	When $n=3$, the only simply-connected 3-manifold is the 3-sphere which has been discussed. So in the following we assume $n>3$.
	
	Let $g$ be a non-degenerate Riemannian metric on $S$. By the discussion after Lemma \ref{correspondence}, the Reeb orbits on the boundary of $U= D_{1}T^{*}S$ all have non-negative Conley-Zehnder indices. Then we pick a family of $S$-shape Hamiltonian functions to be our acceleration data, such that all the lower constant orbits have degrees less than $n+1$. This can be achieved since $U= D_{1}T^{*}S$. Note that the Hamiltonian flow is in the opposite direction of the Reeb flow, in the cylindrical region. From a Reeb orbit to its corresponding Hamiltonian orbit, the index is changed by a sign, plus a error term bounded by one, see Lemma \ref{comparison}. Hence the Conley-Zehnder indices of non-constant Hamiltonian orbits are all less than two. Their degrees, defined as $CZ(\gamma)+n$, are all less than $n+3$ after time-dependent perturbations. So these Hamiltonian functions are index bounded in degree $2n$, since all degree $2n$ generators are upper constant orbits. Then the ignoring upper orbits process says that $SH^{2n}_{M}(D; \Lambda_{E})=0$ for any $E>0$. Finally we apply the Mayer-Vietoris argument in Lemma \ref{com sphere} to complete the proof.
\end{proof}

The disk cotangent bundle can be regarded as a Liouville domain with a smooth Lagrangian skeleton. On the other hand, certain Brieskorn manifolds are Liouville domains of which the Lagrangian skeletons are chains of spheres modeled by trees. In \cite{KvK}, the Reeb orbits of many Brieskorn manifolds have been studied explicitly. Hence one can use the calculation of indices therein to get more applications, like the rigidity of symplectic embeddings of Brieskorn manifolds into Calabi-Yau manifolds.

\subsection{Perturbation of Morse-Bott orbits}
Let $D$ be a Liouville domain with a contact boundary $(C, \alpha)$ in a closed Calabi-Yau manifold $M$ such that $\alpha$ is Morse-Bott index bounded. Then a time-independent $S$-shape Hamiltonian function $H$ has non-degenerate constant orbits, and non-constant one-periodic orbits that are Morse-Bott degenerate, given by the Reeb orbits of $\alpha$. Now we will perturb $H$ to get a non-degenerate $S$-shape Hamiltonian function $H_{t}$, such that it is index bounded. We remark that we are perturbing the Hamiltonian function instead of perturbing the contact form, since the index bounded condition may be destroyed by the later perturbation.

Let $Y$ be the set of $l$-periodic Reeb orbits of $\alpha$. By the Morse-Bott condition $Y$ is a closed smooth submanifold of $C$. It may have several connected components, and we will construct our perturbation component-wisely. For simplicity, we assume that $l=1$. The general case is similar. Our perturbation is a modification of the case of a time-independent Hamiltonian function with transversally non-degenerate orbits, where $Y=S^{1}$.

There is an $S^{1}$-action on $Y$ induced by the Reeb flow $\phi^{t}$. For a Morse function $g:Y \rightarrow \mathbb{R}$, we twist it by the $S^{1}$-action to get a time-dependent function on $Y$:
$$
g_{t}(y):= g(\phi^{1-t}(y)), \quad t\in[0,1], \quad y\in Y.
$$
Next let $N$ be the normal bundle of $Y$ in $M$. We extend $g_{t}$ to be a function $\tilde{g}_{t}$ on $N$ which is supported near the zero section. We also require that $\tilde{g}_{t}$ does not depend on the fiber direction in a small neighborhood of the zero section.

Now for a time-independent $S$-shape Hamiltonian function $H: M\rightarrow \mathbb{R}$, it has degenerate one-periodic orbits in the cylindrical region, which forms the submanifold $Y$. Define $G^{\epsilon}_{t}: S^{1}\times M\rightarrow \mathbb{R}$ as
$$
G^{\epsilon}_{t}(m): H(m) + \epsilon\tilde{g}_{t}(m), \quad m\in M.
$$
Our main result of this subsection is

\begin{proposition}
For small $\epsilon>0$, the one-periodic orbits of $G^{\epsilon}_{t}$ in a small neighborhood of $Y$ are in one-to-one correspondence with critical points of $g$. Let $\gamma$ be a one-periodic orbit of $H$ on $Y$ and $\gamma^{\epsilon}$ be a one-periodic orbits of $G^{\epsilon}_{t}$ near $Y$. We have that
\begin{enumerate}
\item $\int_{\gamma}\theta= \int_{\gamma^{\epsilon}}\theta$ where $\theta$ is the Liouville one-form on $D$;
\item $\abs{CZ(\gamma)- CZ(\gamma^{\epsilon})}\leq \dim_{\mathbb{R}}Y$.
\end{enumerate}
\end{proposition}
\begin{proof}
This proposition, which is known to experts, is a modification of Lemma 2.1 and Proposition 2.2 in \cite{CFHW}.

Let $J$ be a compatible almost complex structure on $M$ which is cylindrical near $C$. By the construction, the Hamiltonian vector field of $G^{\epsilon}_{t}$ is
$$
X_{G^{\epsilon}_{t}}(m)= X_{H}(m) + \epsilon J\nabla \tilde{g}_{t}(m).
$$
Here the gradient is computed with respect to the metric $\omega(\cdot, J\cdot)$. Let $p\in Y$ be a critical point of $g$. Then $\phi^{t}(p)$ is a one-periodic orbit of $H$ on $Y$. We can also check that it is also a one-periodic orbit of $G^{\epsilon}_{t}$. Hence each critical point of $g$ gives a one-periodic orbit of $G^{\epsilon}_{t}$. Next we will show there is no other one-periodic orbit.

Let $U$ be a neighborhood of $Y$ in $M$, which does not contain other one-periodic orbits of $H$ disjoint from $Y$. Then for any open set $V\subset U$ with $Y\subset V$, there exists $\epsilon_{0}>0$ such that for any $0< \epsilon< \epsilon_{0}$ the one-periodic orbits of $G^{\epsilon}_{t}$ in $U$ are also in $V$. This is due to the compactness result in Lemma 2.2 \cite{CFHW}. Hence when $\epsilon$ is small, any one-periodic orbit of $G^{\epsilon}_{t}$ is close to an one-periodic orbit of $H$ on $Y$, particularly in the $W^{1,2}$-topology.

Next consider a nonlinear operator
$$
A: W^{1,2}(S^{1}, N)\rightarrow L^{2}(S^{1}, TN)
$$
given by
$$
A(x(t)):= -J(x'(t)- X_{H}(x(t))).
$$
By the Morse-Bott non-degeneracy, the linearization of $A$ is non-degenerate in the normal direction of $Y$. More precisely, there exists a constant $c>0$ such that for any one-periodic orbit $x_{0}(t)$ of $H$ on $Y$ and a vector field $y(t)$ along $x_{0}(t)$ with $y(t)\notin TY$ for some $t$, we have that
$$
\norm{DA(x_{0})\cdot y(t)} \geq c\norm{y(t)}.
$$

Now define another operator
$$
f: W^{1,2}(S^{1}, N)\rightarrow L^{2}(S^{1}, TN)
$$
given by
$$
f(x(t)):= \tilde{g}_{t}'(x(t)).
$$
Note that the kernel of the operator $A+\epsilon f$ is the set of all one-periodic orbits of $G^{\epsilon}_{t}$ in $N$. Since any one-periodic orbit of $G^{\epsilon}_{t}$ is close to an one-periodic orbit $x_{0}(t)$ of $H$ on $Y$, we can write it as $x_{0}+y(t)$ with a vector field $y(t)$ along $x_{0}(t)$.

Then we use the Taylor expansion to calculate that
$$
(A+ \epsilon f)(x_{0} +y)= A(x_{0}) + DA(x_{0})\cdot y +\epsilon f(x_{0}) +\epsilon Df(x_{0})\cdot y + O(\norm{y}^{2}).
$$
Note that $A(x_{0})=0$ and $\norm{DA(x_{0})\cdot y} \geq c\norm{y}$. So we have that
$$
\begin{aligned}
\norm{(A+ \epsilon f)(x_{0} +y)}&\geq  c\norm{y}+ \epsilon \norm{f(x_{0})} -\epsilon c'\norm{y} + O(\norm{y}^{2})\\
&\geq c''\norm{y} +\epsilon \norm{f(x_{0})}
\end{aligned}
$$
when $\epsilon$ and $\norm{y}$ are sufficiently small. Hence $(A+ \epsilon f)(x_{0} +y)=0$ if and only if $y=0$ and $f(x_{0})=0$, which are the orbits given by critical points of $g$. Geometrically these perturbed orbits are the same orbits which start at critical points of $g$. So the integrations of the Liouville one-form do not change.

The proof of $(2)$ will be a direct computation to relate the Conley-Zehnder index with the Morse index, similar to that in \cite{CFHW}.
\end{proof}

Therefore, given a Liouville domain $D$ with a contact boundary $(C, \alpha)$ in $M$ such that $\alpha$ is Morse-Bott index bounded, we can create non-degenerate index-bounded $S$-shape Hamiltonian functions associated with $D$. Then we can use them to construct the spectral sequence as we did in the Morse index bounded case.

\bibliographystyle{amsplain}

\end{document}